\documentclass[12pt,twoside]{amsart}
\usepackage{a4wide, amsmath,amsbsy,amsfonts,amssymb,mathabx,
stmaryrd,amsthm,mathrsfs,graphicx, amscd, tikz-cd, cancel}
\usepackage[all]{xy}
\usepackage[normalem]{ulem}
\usetikzlibrary{matrix,arrows,decorations.pathmorphing}
\usepackage[active]{srcltx}
\usepackage[
	hypertexnames=false,
	hyperindex,
	pagebackref,
	pdftex,
	breaklinks=true,
	bookmarks=false,
	colorlinks,
	linkcolor=blue,
	citecolor=red,
	urlcolor=red,
]{hyperref}
\usepackage[mathscr]{eucal}
\DeclareMathSymbol{\Alpha}{\mathalpha}{operators}{"41}
\DeclareMathSymbol{\Beta}{\mathalpha}{operators}{"42}
\DeclareMathSymbol{\Epsilon}{\mathalpha}{operators}{"45}
\DeclareMathSymbol{\Zeta}{\mathalpha}{operators}{"5A}
\DeclareMathSymbol{\Eta}{\mathalpha}{operators}{"48}
\DeclareMathSymbol{\Iota}{\mathalpha}{operators}{"49}
\DeclareMathSymbol{\Kappa}{\mathalpha}{operators}{"4B}
\DeclareMathSymbol{\Mu}{\mathalpha}{operators}{"4D}
\DeclareMathSymbol{\Nu}{\mathalpha}{operators}{"4E}
\DeclareMathSymbol{\Omicron}{\mathalpha}{operators}{"4F}
\DeclareMathSymbol{\Rho}{\mathalpha}{operators}{"50}
\DeclareMathSymbol{\Tau}{\mathalpha}{operators}{"54}
\DeclareMathSymbol{\Chi}{\mathalpha}{operators}{"58}
\DeclareMathSymbol{\omicron}{\mathord}{letters}{"6F}

\newcommand\la{\langle}
\newcommand\ra{\rangle}

\newcommand\Z{{\mathbb Z}}
\newcommand\ZZ{{\widehat{\mathbb Z}}}

\newcommand\Q{{\mathbb Q}}
\newcommand\F{{\mathbb F}}

\newcommand\N{{\mathbb N}}

\newcommand\cL{{\mathscr L}}

\newcommand\Ra{\Rightarrow}

\newcommand\aut{\operatorname{Aut}}

\newcommand\sky{\operatorname{skysc}}
\newcommand\cosky{\operatorname{coskysc}}

\renewcommand\mod{\operatorname{Mod}}

\newcommand\map{\operatorname{Map}}

\newcommand\hookra{\hookrightarrow}
\newcommand\tura{\twoheadrightarrow}
\newcommand\da{\downarrow}

\renewcommand{\hom}{\operatorname{Hom}}
\newcommand\chom{{\mathscr Hom}}
\renewcommand{\Im}{\operatorname{Im}} 
\newcommand\sr{\stackrel}

\newcommand\ssm{\smallsetminus}
\newcommand\ol{\overline}
\newcommand\ul{\underline}

\newcommand\cE{\mathscr E}
\newcommand\cM{{\mathscr M}}

\newcommand\cH{{\mathscr H}}
\newcommand\cA{{\mathscr A}}
\newcommand\cB{{\mathscr B}}
\newcommand\cC{{\mathscr C}}

\newcommand\bH{\overline{H}}
\newcommand\bT{\overline{T}}

\newcommand\cF{{\mathcal F}}

\newcommand\cS{{\mathscr S}}
\newcommand\cU{{\mathscr U}}
\newcommand\cT{{\mathcal T}}

\newcommand\cK{{\mathscr K}}

\newcommand\cG{{\mathscr G}}

\newcommand\ld{\lambda}
\newcommand\Ld{\Lambda}
\newcommand\wh{\widehat}

\newcommand\td{\tilde}
\newcommand\sg{\sigma}

\newcommand\bt{\bullet}

\newcommand\ind{\operatorname{Ind}}
\newcommand\coind{\operatorname{Coind}}
\newcommand\id{\operatorname{id}}

\def\co{\colon\thinspace}

\newcommand\Rep{\operatorname{Rep}}

\newcommand\ad{\operatorname{ad}}

\newcommand\Sh{\operatorname{Shv}}
\newcommand\Et{\acute{\mathrm{E}}\mathrm{t}}
\newcommand\CoEt{\mathrm{Co}\acute{\mathrm{E}}\mathrm{t}}

\newcommand\Op{\operatorname{Op}}
\newcommand\Ob{\operatorname{Ob}}

\newcommand\rk{\operatorname{rk}}

\DeclareMathOperator*{\colim}{colim}

\def\mr{\mathrm}

\newtheorem{innercustomthm}{Theorem}
\newenvironment{customtheorem}[1]
  {\renewcommand\theinnercustomthm{#1}\innercustomthm}
  {\endinnercustomthm}

\newtheorem{theorem}{Theorem}[section]
\newtheorem{corollary}[theorem]{Corollary}
\newtheorem{proposition}[theorem]{Proposition}

\newtheorem{lemma}[theorem]{Lemma}

\theoremstyle{definition}
\newtheorem{definition}[theorem]{Definition}   
\newtheorem{remark}[theorem]{Remark}

\begin{document}

\title[Lannes' $T$-functor and mod-$p$ cohomology of profinite groups]{Lannes' $T$-functor and mod-$p$ cohomology\\ of profinite groups}

\author{Marco Boggi}
\address{UFF - Instituto de Matem\'atica e Estat\'{\i}stica -
Niter\'oi - RJ 24210-200; Brazil}
\email{marco.boggi@gmail.com}

\begin{abstract}The Lannes-Quillen theorem relates the mod-$p$ cohomology of a finite group $G$ with the mod-$p$ cohomology of centralizers
of abelian elementary $p$-subgroups of $G$, for $p>0$ a prime number. This theorem was extended to profinite groups whose mod-$p$ 
cohomology algebra is finitely generated by Henn. In a weaker form, the Lannes-Quillen theorem was then extended by Symonds to arbitrary 
profinite groups.
Building on Symonds' result, we formulate and prove a full version of this theorem for all profinite groups. For this purpose, we develop a theory 
of products for families of discrete torsion modules, parameterized by a profinite space\footnotemark, which is dual, in a very precise sense, 
to the theory of coproducts for families of profinite modules, parameterized by a profinite space, developed by Haran, Melnikov and Ribes.
In the last section, we give applications to the problem of conjugacy separability of $p$-torsion elements and finite $p$-subgroups.
\end{abstract}
\footnotetext{The day after this article was posted on the arXiv, the paper \cite{Wilkes}, which develops
a similar theory, although with a different approach, was also posted there. 
In Section~\ref{comparison}, we give a detailed account of how the classical Haran, Melnikov, Ribes 
and the new \cite[Wilkes]{Wilkes} theories are related to ours.} 
\maketitle

\section{Introduction}
Throughout the paper $p$ will be a fixed positive prime number. 
Homology (or better cohomology) decompositions are formulas which relate the mod-$p$ cohomology of a group $G$ with that of a set 
of proper subgroups. Typically, these subgroups are the normalizers or centralizers of some special class of "simple" subgroups of the group 
$G$ (e.g.\ elementary abelian $p$-subgroups). In this paper, we will consider one of these formulas which is obtained combining the theory of 
Lannes' $T$-functor with a classical Quillen homology decomposition in terms of centralizers of elementary abelian $p$-subgroups 
(cf.\ \cite{DH}, \cite{Lannes0} and \cite{Quillen}).

\subsection{Lannes' $T$-functor}
Let $\cU$ be the category of unstable modules over the mod-$p$ Steenrod algebra (cf.\ \cite[Definition~2 and~3]{DH}) and $\cK$ the category
of unstable algebras over the mod-$p$ Steenrod algebra (cf.\ \cite[Definition~4]{DH}). For a topological space $X$, let 
$H^\bullet(X):=H^\bullet(X;\F_p)$ be its mod-$p$ cohomology algebra. Then, $H^\bullet(X)\in\cK$. For an abstract group $G$, 
we let $H^\bullet(G):=H^\bullet(BG;\F_p)$, where $BG$ is the classifying space of $G$. 

For a profinite group $G$, we let $H^\bullet(G):=H^\bullet(BG;\F_p)$, where $BG$ is the profinite classifying space of $G$ 
in the sense of \cite{Quick1} and so $H^\bullet(G)$ is the continuous cohomology algebra of $G$ with constant $\F_p$-coefficients.
There still holds $H^\bullet(G)\in\cK$.

Let $V$ be an elementary abelian $p$-group. \emph{Lannes' $T$-functor} $T_V\co\cU\to\cU$ is the left adjoint of the functor 
$H^\bullet(V)\otimes\_\co\cU\to\cU$, that is to say it is characterized by the identity, for all $M,N\in\cU$:
\[\hom_\cU(T_V M,N)=\hom_\cU(M,H^\bullet(V)\otimes N).\]
Since $H^\bullet(V)$ is an algebra of finite type (i.e.\ it is finite in all degrees), the functor $H^\bullet(V)\otimes\_$ preserves small limits and 
the existence of Lannes' $T$-functor $T_V$ follows from the Special Adjoint Functor Theorem
(cf.\ \cite[Corollary to Theorem~2, Section~8, Ch.~V]{MacLane}). 

For a topological space $X$, let $\cC(BV,X)$ be the space of continuous maps. The evaluation map $\mathrm{ev}_X\co BV\times\cC(BV,X)\to X$ 
induces on cohomology a map:
\[\mathrm{ev}_X^\ast\co H^\bullet(X)\to H^\bullet(V)\otimes H^\bullet(\cC(BV,X)).\]       
By the adjunction formula, we then obtain a natural homomorphism: 
\[\ad(\mathrm{ev}_X^\ast)\co T_V H^\bullet(X)\to H^\bullet(\cC(BV,X)).\] 

It turns out that, in many cases, the map $\ad(\mathrm{ev}_X^\ast)$ is an isomorphism. Thus, for a general $M\in\cU$, we may think of
the module $T_V M$ as a sort of categorical avatar of the mod-$p$ cohomology algebra of the mapping space $\cC(BV,X)$. 

An important property of Lannes' $T$-functor\! (cf.\ \cite[Proposition~2.4.5]{Lannes}, \cite[Theorem~3.8.1]{Schwartz} and \cite[Theorem~14]{DH}) 
is that it restricts to a functor $T_V\co\cK\to\cK$ and that, for all $M,N\in\cK$, we still have the adjunction formula:
\[\hom_\cK(T_V M,N)=\hom_\cK(M,H^\bullet(V)\otimes N).\]
Other fundamental properties of Lannes' $T$-functor $T_V$ are (cf.\ \cite{Lannes}, \cite{DH} and \cite{Schwartz}):
\begin{itemize}
\item $T_V$ is exact;
\item for all $M,N\in\cU$, there is a natural isomorphism $T_V(M\otimes N)\cong T_V M\otimes T_V N$;
\item if $M\in\cU$ is finite, then $T_V M=M$. More generally, $T_V M=M$, if $M$ is the union of its finite unstable submodules.
\end{itemize}

\subsection{The reduced $T$-functor}\label{defreduced}
There is a reduced version of Lannes' $T$-functor which will turn out to be particularly useful for some of our applications. Let $\bH^\bt(\_)$ 
be the reduced cohomology functor with constant $\F_p$-coefficients. This also takes values in the category $\cU$ of unstable modules 
over the mod-$p$ Steenrod algebra and there is a natural isomorphism $H^\bt(\_)\cong\F_p\oplus\bH^\bt(\_)$. 
For an elementary abelian $p$-group $V$, the \emph{reduced $T$-functor} $\bT_V$ is then defined to be the left adjoint of the functor 
$\bH^\bullet(V)\otimes\_\co\cU\to\cU$. From the various definitions involved, it immediately follows that, for $M\in\cU$, 
there is a natural isomorphism of unstable modules:
\begin{equation}\label{reducedT}
T_V M\cong M\oplus \bT_V M.
\end{equation}

\subsection{Lannes' generalization of Quillen's theorem}
For a group $G$, we denote by $\Rep(V,G)=\hom(V,G)/G$ the set of $G$-conjugacy classes of homomorphisms $V\to G$ and, 
for a subgroup $H$ of $G$, by $C_G(H)$ the centralizer of $H$ in $G$. 

Let us denote a conjugacy class in $\Rep(V,G)$ by $[\rho]$, where $\rho\co V\to G$ is a chosen representative. 
To the homomorphism $\rho\co V\to G$, we associate the homomorphism:
\[c_\rho\co V\times C_G(\rho(V))\to G\]
defined by the assignment $(v,x)\mapsto\rho(v)x$. This induces a homomorphism of unstable algebras 
$c_\rho^\ast\co H^\bullet(G)\to H^\bullet(V)\otimes H^\bullet(C_G(\rho(V)))$ and then the adjoint homomorphism of unstable algebras:
\[\ad(c_\rho^\ast)\co T_V H^\bullet(G)\to H^\bullet(C_G(\rho(V))).\]
Taking the direct product of these maps over the set $\Rep(V,G)$, we get a map:
\begin{equation}\label{Lannesiso}
\prod_{[\rho]\in\Rep(V,G)}\ad(c_\rho^\ast)\co T_V H^\bullet(G)\to\prod_{[\rho]\in\Rep(V,G)} H^\bullet(C_G(\rho(V))).
\end{equation}
There holds (cf.\ \cite[\emph{Exemples} in Section~3.4]{Lannes} and \cite[Theorem~19]{DH}):

\begin{theorem}[Lannes]\label{Lannesfinite}For a finite group $G$ and an elementary abelian $p$-group $V$, the map~\eqref{Lannesiso}
is an isomorphism of unstable algebras over the mod-$p$ Steenrod algebra.
\end{theorem}

\begin{definition}\label{Lannesgroup}A group for which the map~\eqref{Lannesiso} is an isomorphism is called a \emph{Lannes group}.
\end{definition}

\begin{remark}\label{Lannesgeneral}The following is a list of Lannes groups: 
\begin{enumerate}
\item Groups of finite virtual $p$-cohomological dimension (cf.\ \cite{Lannes0} and \cite{DH}).
\item Groups which act on a finite dimensional $CW$-complex cocompactly and with finite stabilizers (cf.\ \cite{H1}).
\item Compact Lie groups (cf.\ \cite{Lannes0} and \cite{DH}).
\item Profinite groups with finitely generated mod-$p$ (continuous) cohomology algebra (cf.\ \cite{H2}).
\end{enumerate}
\end{remark}

\subsection{A generalization to arbitrary profinite groups}
For a profinite group $G$ with infinitely many conjugacy classes of elementary abelian $p$-subgroups, we cannot expect the
map of Theorem~\ref{Lannesfinite} (which is still well defined) to be an isomorphism. For one thing, in this case, $T_V H^\bullet(G)$
is naturally endowed with the discrete topology while the natural topology on the product $\prod_{[\rho]\in\Rep(V,G)} H^\bullet(C_G(\rho(V)))$
is prodiscrete. However, this only happens because the abstract product of modules is not the right notion to consider here. 

Symonds (cf.\ \cite[Proposition~4.1]{Symonds}) proved that, for an arbitrary profinite group $G$, the map of Theorem~\ref{Lannesfinite} has dense image.
In this paper, thanks to a topological notion of product for certain families of discrete $\F_p$-modules parameterized by a profinite space $T$, 
we are able to give a more precise result. We will call such a family $\{M_t\}_{t\in T}$ an \emph{\'etal\'e space of torsion discrete $\F_p$-modules
on the profinite space $T$} (cf.\ Section~\ref{definitionetdis}, for the precise definition) and we will denote by $\prod_{t\in T}^\mathrm{top}M_t$ 
the corresponding product (cf.\ Section~\ref{defprodfunctor}, for the precise definition). The main result of the paper then is:

\begin{customtheorem}{A}For a profinite group $G$ and an elementary abelian $p$-group $V$, let us denote by $\pi\co\hom(V,G)\to\Rep(V,G)$ 
the natural orbit map. There is then a natural isomorphism of discrete $\F_p$-vector spaces:
\[T_V H^\bullet(G)\cong\sideset{}{^{\mathrm{top}}}\prod_{[\rho]\in\Rep(V,G)}H^\bullet(G;\cC(\pi^{-1}([\rho]),\F_p)),\]
where $H^\bullet(G;\cC(\pi^{-1}([\rho]),\F_p))\cong H^\bullet(C_G(\rho(V)))$, for all $[\rho]\in\Rep(V,G)$.
\end{customtheorem}

\subsection{Applications to conjugacy separability}
Let us recall that an element $x$ of a group $G$ is \emph{conjugacy distinguished} if, whenever an element $y\in G$ is not conjugated 
to $x$, there is a finite quotient $G/N$ of $G$ such that the images of $x$ and $y$ in $G/N$ are also not conjugated. For a residually finite
group $G$, this is equivalent to the property that the conjugacy class of $x$ in $G$ is closed for the profinite topology, that is to say, 
for the topology induced by the embedding of $G$ in its profinite completion $\wh{G}$. 

Similarly, we then say that a finite subgroup $H$ of $G$ is \emph{subgroup conjugacy distinguished} if the conjugacy class $H^G$ of $H$ is closed 
in the space of finite subgroups $\cS_f(G)$ of $G$, where we endow the set $\cS_f(G)$ with the topology induced by its embedding 
in the profinite space of all closed subgroups of the profinite completion $\wh{G}$ of $G$.

Theorem~A then implies the following conjugacy separability criterion for $p$-elements and elementary abelian $p$-subgroups of a Lannes group
(cf.\ Corollary~\ref{separability2}) which improves a criterion by Minasyan and Zalesskii (cf.\ \cite[Theorem~1.5]{MZ}).
We use the notation $i\gg 0$ in the meaning of "for all $i$ sufficiently large":

\begin{customtheorem}{B}Let $G$ be a Lannes group such that the natural homomorphism $\iota\co G\to\wh{G}$ to its profinite completion 
induces a surjective map $\iota^\ast\co H^i(\wh{G};\F_p)\to H^i(G;\F_p)$ for $i\gg 0$, then $p$-elements of $G$ 
are conjugacy distinguished and elementary abelian $p$-subgroups of $G$ are subgroup conjugacy distinguished.
\end{customtheorem}

For a group $G$, let $\cS_p(G)$ be the category of finite $p$-subgroups of $G$, where the morphisms are induced by inclusion and conjugation by elements 
of $G$. Another application (which one should compare with \cite[Theorem~1.1]{Symonds} and \cite[Theorem~A]{BZ}) of Theorem~A is then:

\begin{customtheorem}{C}Let $G$ be a Lannes group such that the cohomology ring $H^\bt(G;\F_p)$ is a finitely generated $\F_p$-algebra and 
let $f\co G\to L$ be a homomorphism to a profinite group such that the induced map $f^\ast\co H^i(L;\F_p)\to H^i(G;\F_p)$ is an isomorphism 
for $i\gg 0$. Then, the induced functor $f_p\co\cS_p(G)\to\cS_p(L)$ is an equivalence of categories.
\end{customtheorem}
\medskip

\noindent
{\bf Acknowledgements.} I thank two anonymous referees for their comments on preliminary versions of this manuscript.

\section{Sheaves and \'etal\'e spaces of profinite and discrete torsion modules}\label{algdef}
Let us fix an involutive profinite $\ZZ$-algebra $\Rho$, where we let $\sg$ be the involution and denote by $r^\ast:=\sg(r)$ the adjoint of 
$r\in\Rho$. We assume moreover that $\Rho$ is endowed with an augmentation map $\Rho\tura R$ onto the image $R$ of the structure map 
$\ZZ\to\Rho$ (typically, $\Rho=\ZZ$, $\Z_p$, $\F_p$, with trivial involution, or the group algebra of a profinite group $G$ over one of the 
previous rings, with involution the linear extension of the map $g\mapsto g^{-1}$, for $g\in G$).

Let then $\Rho\text{-}\mod^\mathrm{pro}$ and $\Rho\text{-}\mod^\mathrm{dis}$ be, respectively, the categories of (topological, left) profinite, 
discrete (and then torsion) $\Rho$-modules. They are both abelian categories and intersect in the subcategory
$\Rho\text{-}\mod^\mathrm{fin}=\Rho\text{-}\mod^\mathrm{pro}\cap\Rho\text{-}\mod^\mathrm{dis}$ of finite discrete $\Rho$-modules of which
they are, respectively, the completion and the cocompletion.

In this paper, a topological space is \emph{compact} if every open cover has a finite subcover. 
A \emph{profinite space} is a cofiltered limit of finite discrete spaces, endowed with the limit topology. 
Equivalently, with the above conventions, a profinite space is a totally disconnected,  compact, Hausdorff space. 

The following well-known properties of profinite spaces will be used all through the paper (cf.\ \cite[Theorem~1.1.12]{RZ}). 
A profinite space $X$ admits a base for the topology which consists of simultaneously closed and open (briefly, \emph{clopen}) subsets. 
In particular, every open cover of $X$ can be refined by a cover consisting of finitely many disjoint clopen subsets.

\subsection{Sheaves of profinite and discrete torsion modules over a profinite space}
Let $T$ be a profinite space and let $\Op(T)$ be the category of open subsets of $T$.
A \emph{presheaf of finite discrete $\Rho$-modules on $T$} is a functor $\cF\co\Op(T)^\mathrm{op}\to\Rho\text{-}\mod^\mathrm{fin}$.
A \emph{sheaf of finite discrete $\Rho$-modules on $T$} is such a presheaf which satisfies the condition that,
for every open covering $\{U_i\}_{i\in I}$ of an open subset $U$ of $T$, the following diagram is an equalizer:
\[\cF(U)\to\prod_{i\in I}\cF(U_i)\rightrightarrows\prod_{i,j\in I}\cF(U_i\cap U_j).\]

Let $\cS^\mathrm{fin}_T(\Rho)$ be the category of \emph{locally constant} sheaves of finite discrete $\Rho$-modules on $T$, 
which we regard as a subcategory of the category $\Sh^\mathrm{top}_T(\Rho)$ of sheaves of topological $\Rho$-modules on $T$.
We then let the categories $\cS^\mathrm{pro}_T(\Rho)$ and $\cS^\mathrm{dis}_T(\Rho)$ be respectively the completion and 
the cocompletion of $\cS^\mathrm{fin}_T(\Rho)$ inside the category $\Sh^\mathrm{top}_T(\Rho)$. They consist, respectively,
of sheaves of profinite and discrete $\Rho$-modules.

The stalk $\cF_x$ of the colimit $\cF$ of a filtered diagram $\{\cF_\alpha\}_{\alpha\in\Alpha}$ of sheaves is naturally isomorphic to 
$\colim_{\alpha\in\Alpha}(\cF_\alpha)_x$. The analogue is not true in general for cofiltered limits of sheaves. 
The following lemma shows that the category of \'etal\'e spaces over $T$ is better behaved with respect to the operation 
of completion:

\begin{lemma}\label{stalkspro}For $\{\cF_\alpha\}_{\alpha\in\Alpha}$ a cofiltered diagram in $\cS^\mathrm{fin}_T(\Rho)$,
let $\{(F_\alpha,\pi_\alpha)\}_{\alpha\in\Alpha}$ be the cofiltered diagram of associated \'etal\'e spaces and $(F,\pi)$ its limit.
Then, for all $x\in T$, there is a natural isomorphism $F_x:=\pi^{-1}(x)\cong\lim_{\alpha\in\Alpha}(\cF_\alpha)_x$.
\end{lemma}

\begin{proof}For every $\alpha\in\Alpha$, there is a natural isomorphism $\cF_\alpha\cong F_\alpha$ and then, passing to the limit,
a natural isomorphism $F_x=\lim_{\alpha\in\Alpha}(F_\alpha)_x=\lim_{\alpha\in\Alpha}(\cF_\alpha)_x$.
\end{proof}

\begin{remark}\label{stalkscomp}Let $\cF:=\lim_{\alpha\in\Alpha}\cF_\alpha\in\cS^\mathrm{pro}_T(\Rho)$. 
The sheaf $\cF$ then identifies with the sheaf of continuous sections of the natural map $\pi\co F\to T$. This implies that,
for all $x\in T$, there is a natural homomorphism $\cF_x\to F_x$. However, in general, this map is neither injective nor surjective.
In the special case when $\cF$ is also locally constant, it is easy to see that the map $\cF_x\to F_x$ is indeed an isomorphism.
It is also possible to show that, in case the transition maps of the cofiltered diagram $\{\cF_\alpha\}_{\alpha\in\Alpha}$ are surjective,
the map $\cF_x\to F_x$ is surjective but yet, even in this case, this is not necessarily an isomorphism. We omit the rather nasty details, 
since these observations are not going to be used in the paper.
\end{remark}

\subsection{Pontryagin duality for sheaves of profinite and discrete torsion modules}\label{Pontduality1}
For $A\in\Rho\text{-}\mod^\mathrm{pro}$ or $\Rho\text{-}\mod^\mathrm{dis}$, let $A^\vee:=\hom_\ZZ(A,\Q/\Z)$ and define a structure 
of (topological, left) $\Rho$-module on $A^\vee$ by setting $r\cdot f(a)=f(r^\ast\cdot a)$, for $r\in\Rho$, $f\in A^\vee$ and $a\in A$. 
This is called the \emph{Pontryagin dual of $A$}. It is well known (cf.\ \cite[Section~5.1]{RZ}) 
that the Pontryagin duality functor $\hom_\ZZ(\_,\Q/\Z)$ is a duality of categories between the abelian categories
$\Rho\text{-}\mod^\mathrm{pro}$ and $\Rho\text{-}\mod^\mathrm{dis}$. 

For a $\Rho$-module $A$, let us denote by $\ul{A}_T$ (or more simply by $\ul{A}$) the constant sheaf on $T$ associated to $A$.
Then, for a sheaf $\cF\in\Sh^\mathrm{top}_T(\Rho)$, we define the \emph{Pontryagin dual of $\cF$} to be the Hom sheaf:
\[\cF^\vee:=\chom_{\ul{\ZZ}}(\cF,\ul{\Q/\Z})\in\Sh^\mathrm{top}_T(\Rho).\]

\begin{lemma}\label{locallyconstanthom}For $\cF\in\cS^\mathrm{fin}_T(\Rho)$ and $\cG\in\cS^\mathrm{dis}_T(\Rho)$ locally constant,
the sheaf $\chom_{\ul{\Rho}}(\cF,\cG)$ is also locally constant and, for all $t\in T$, there holds 
$\chom_{\ul{\Rho}}(\cF,\cG)_t=\hom_\Rho(\cF_t,\cG_t)$.
\end{lemma}

\begin{proof}Let $U$ be an open subset of $T$ such that both $\cF|_U$ and $\cG|_U$ are constant sheaves and let 
$M,N\in\Rho\text{-}\mod^\mathrm{dis}$, with $M$ finite, be such that $\ul{M}_U=\cF|_U$ and $\ul{N}_U=\cG|_U$.
We then have the identities:
\[\chom_{\ul{\Rho}}(\cF,\cG)|_U=\chom_{\ul{\Rho}_U}(\cF|_U,\cG|_U)=\chom_{\ul{\Rho}_U}(\ul{M}_U,\ul{N}_U).\]

The intersection of the image of the natural map $\ul{\hom_\ZZ(M,N)}_U\to\chom_{\ul{\ZZ}_U}(\ul{M}_U,\ul{N}_U)$ 
with $\chom_{\ul{\Rho}_U}(\ul{M}_U,\ul{N}_U)\subseteq\chom_{\ul{\ZZ}_U}(\ul{M}_U,\ul{N}_U)$ obviously
coincides with the image of the natural map $\ul{\hom_\Rho(M,N)}_U\to\chom_{\ul{\Rho}_U}(\ul{M}_U,\ul{N}_U)$. 

Since $M$ is finite, it is of finite presentation as a $\ZZ$-module and,
by \cite[\href{https://stacks.math.columbia.edu/tag/093P}{Tag 093P}, (1), Lemma~18.43.4]{stacks-project},
the map $\ul{\hom_\ZZ(M,N)}_U\to\chom_{\ul{\ZZ}_U}(\ul{M}_U,\ul{N}_U)$ is an isomorphism. By the above remark,
this implies that the map $\ul{\hom_\Rho(M,N)}_U\to\chom_{\ul{\Rho}_U}(\ul{M}_U,\ul{N}_U)$ is an isomorphism as well, 
which implies both claims of the lemma.
\end{proof}

\begin{lemma}\label{locstronglyfiltered}\leavevmode\begin{enumerate}
\item Let $\{\cF_\alpha\}_{\alpha\in\Alpha}$ be a cofiltered diagram in $\cS^\mathrm{fin}_T(\Rho)$ with 
surjective transition maps. Then the limit $\cF:=\lim_{\alpha\in\Alpha}\cF_\alpha\in\cS^\mathrm{pro}_T(\Rho)$ has the property
that, for every epimorphism of sheaves $p\co\cF\to\cK$, where $\cK\in\cS^\mathrm{fin}_T(\Rho)$, there is an $\alpha\in\Alpha$
such that $p$ factors through the natural epimorphism $\cF\to\cF_\alpha$.
\item Let $\{\cF_\alpha\}_{\alpha\in\Alpha}$ be a filtered diagram in $\cS^\mathrm{fin}_T(\Rho)$ with 
injective transition maps. Then its colimit $\cF:=\colim_{\alpha\in\Alpha}\cF_\alpha\in\cS^\mathrm{dis}_T(\Rho)$ has the property
that every subsheaf $\cK\subseteq\cF$, where $\cK\in\cS^\mathrm{fin}_T(\Rho)$, lies in the image of $\cF_\alpha$, for some $\alpha\in\Alpha$.
\end{enumerate}
\end{lemma}

\begin{proof}The proofs of the two items are similar, so we only prove the first item. 
By Lemma~\ref{stalkspro}, for all $x\in T$, there is an isomorphism $\cF_x=\lim_{\alpha\in\Alpha}(\cF_\alpha)_x$.
This implies that, for all $x\in T$, the induced epimorphism $\cF_x\to\cK_x$ factors through some natural epimorphism $\cF_x\to(\cF_{\phi(x)})_x$,
for some function $\phi\co T\to\Alpha$. Since the sheaf $\cK$ on $T$ is locally constant, we can assume that the function $\phi$ is also locally 
constant. Since $T$ is compact, the image of $\phi$ in $\Alpha$ then admits an upper bound $\alpha\in\Alpha$ and the conclusion follows.
\end{proof}

The following general topology lemma will be useful:

\begin{lemma}\label{lochomcomphaus}A surjective local homeomorphism between compact Hausdorff spaces is a covering map.
\end{lemma}

\begin{proof}Not difficult to prove and a particular case of the more general \cite[Lemma~2]{Ho}.
\end{proof}

We say that a sheaf of discrete sets is \emph{compact} (resp.\ \emph{Hausdorff}) if its associated \'etal\'e space is 
compact (resp.\ Hausdorff). We then have:

\begin{lemma}\label{compactsubsheaf}A compact subsheaf of a locally constant sheaf over a compact Hausdorff space is locally constant.
\end{lemma}

\begin{proof}A locally constant sheaf $\cF$ on a Hausdorff space $T$ is Hausdorff, so that a compact subsheaf $\cK$ of $\cF$ 
is also Hausdorff. The associated \'etal\'e space $K$ is then a compact Hausdorff space and the natural map $\pi\co K\to T$ 
is a local homeomorphism between compact Hausdorff spaces and so, by Lemma~\ref{lochomcomphaus}, a covering map.
\end{proof}   

\begin{remark}\label{subsheaffinite}By Lemma~\ref{compactsubsheaf}, a compact subsheaf of 
a sheaf in $\cS^\mathrm{fin}_T(\Rho)$ is also in $\cS^\mathrm{fin}_T(\Rho)$.
\end{remark}

\begin{lemma}\label{compactsubquotient}Every $\cF\in\cS^\mathrm{pro}_T(\Rho)$ is the limit of a cofiltered diagram 
in $\cS^\mathrm{fin}_T(\Rho)$ with surjective transition maps.
\end{lemma}

\begin{proof}Let $\cF$ be the limit of a cofiltered diagram $\{\cF_\alpha\}_{\alpha\in\Alpha}$ and let $\ol{\cF}_\alpha$ be the image of 
the natural map of sheaves $\cF\to\cF_\alpha$, for all $\alpha\in\Alpha$. By Remark~\ref{subsheaffinite}, 
we have that $\ol{\cF}_\alpha\in\cS^\mathrm{fin}_T(\Rho)$, for all $\alpha\in\Alpha$.
Since $\cF=\lim_{\alpha\in\Alpha}\ol{\cF}_\alpha$, the conclusion follows.
\end{proof}

We can now prove:

\begin{proposition}\label{Pontryagindualitydef}For $\cF\in\cS^\mathrm{pro}_T(\Rho)$ (resp.\ $\cF\in\cS^\mathrm{dis}_T(\Rho)$),
we have that $\cF^\vee\in\cS^\mathrm{dis}_T(\Rho)$ (resp.\ $\cF^\vee\in\cS^\mathrm{pro}_T(\Rho)$). 
\end{proposition}

\begin{proof}By Lemma~\ref{compactsubquotient}, $\cF\in\cS^\mathrm{pro}_T(\Rho)$ is the limit of a cofiltered diagram 
$\{\cF_\alpha\}_{\alpha\in\Alpha}$ in $\cS^\mathrm{fin}_T(\Rho)$ with surjective transition maps. By Lemma~\ref{compactsubsheaf}, 
(i) of Lemma~\ref{locstronglyfiltered} and Lemma~\ref{locallyconstanthom}, we then have:
\[\chom_{\ul{\ZZ}}(\cF,\ul{\Q/\Z})\cong\colim_{\alpha\in\Alpha}\chom_{\ul{\ZZ}}(\cF_\alpha,\ul{\Q/\Z})\in\cS^\mathrm{dis}_T(\Rho).\]

Similarly, a sheaf $\cF\in\cS^\mathrm{dis}_T(\Rho)$ is the filtered colimit of a diagram $\{\cF_\alpha\}_{\alpha\in\Alpha}$ 
in $\cS^\mathrm{fin}_T(\Rho)$, so that we have:
\[\chom_{\ul{\ZZ}}(\cF,\ul{\Q/\Z})\cong\lim_{\alpha\in\Alpha}\chom_{\ul{\ZZ}}(\cF_\alpha,\ul{\Q/\Z})\in\cS^\mathrm{pro}_T(\Rho).\]
\end{proof}

We have thus defined two functors:
\[\Phi_T\co\cS^\mathrm{pro}_T(\Rho)\to\cS^\mathrm{dis}_T(\Rho)^\mathrm{op}\hspace{0.3cm}\mbox{and}
\hspace{0.3cm}\Psi_T\co\cS^\mathrm{dis}_T(\Rho)\to\cS^\mathrm{pro}_T(\Rho)^\mathrm{op}.\]
It is easy to check that there are natural isomorphisms of functors $\Psi_T^\mathrm{op}\circ\Phi_T\Ra\id_{\cS^\mathrm{pro}_T(\Rho)}$ and 
$\Phi_T^\mathrm{op}\circ\Psi_T\Ra\id_{\cS^\mathrm{dis}_T(\Rho)}$, so that we have:

\begin{theorem}\label{Pontryagindualitysheaves}The functors $\Phi_T$ and $\Psi_T$ establish a duality between the 
abelian categories $\cS^\mathrm{pro}_T(\Rho)$ and $\cS^\mathrm{dis}_T(\Rho)$.
\end{theorem}

\begin{remark}\label{commlimits}Note that the functor $\Psi_T$ commutes with filtered colimits but, in general, the functor $\Phi_T$ 
only commutes with limits of cofiltered diagrams with surjective transition maps (cf.\ the proof of Proposition~\ref{Pontryagindualitydef}).
\end{remark}

An immediate consequence of Theorem~\ref{Pontryagindualitysheaves} and Lemma~\ref{compactsubquotient} is the following:

\begin{corollary}\label{completion}Every sheaf in $\cS^\mathrm{dis}_T(\Rho)$ is a filtered colimit of sheaves in $\cS^\mathrm{fin}_T(\Rho)$
along injective transition maps. 
\end{corollary}

We can then determine more precisely the topology of the sheaves in $\cS^\mathrm{dis}_T(\Rho)$:

\begin{proposition}\label{locallycompact}The \'etal\'e space of a sheaf in $\cS^\mathrm{dis}_T(\Rho)$ 
is a totally disconnected locally compact Hausdorff space.
\end{proposition}

\begin{proof}Let $(F,\pi)$ be \'etal\'e space of a sheaf $\cF\in\cS^\mathrm{dis}_T(\Rho)$. By Corollary~\ref{completion},
the space $F$ is a filtered colimit of totally disconnected compact Hausdorff spaces along injective transition maps.
By \cite[Lemma~3.3]{Strickland}, it then follows that $F$ is a compactly generated weakly Hausdorff space 
(cf.\ \cite[Definition~1.1 and Definition~1.2]{Strickland}).

The profinite space $T$ has a base for the topology which consists of clopen subsets. The fact that $F$ is compactly generated weakly Hausdorff 
then implies that the image of a clopen subset of $T$ by a local continuous section of the natural map $\pi\co F\to T$ is closed in $F$. 
Since the \'etal\'e space $F$ is locally homeomorphic to $T$, it follows that we can separate points in the fibers of $\pi$ by means of compact
clopen profinite subsets and then that $F$ is Hausdorff, locally compact and totally disconnected.  
\end{proof}

\subsection{The category of pro\'etal\'e spaces of profinite $\Rho$-modules over $T$}\label{definitionetpro}
To a cofiltered diagram $\{\cF_\alpha\}_{\alpha\in\Alpha}$ in $\cS^\mathrm{fin}_T(\Rho)$ with limit
$\cF\in\cS_T^\mathrm{pro}(\Rho)$, we can associate the limit $(F,\pi)$, in the category of topological spaces over $T$, of the corresponding
cofiltered diagram of \'etal\'e spaces $\{(F_\alpha,\pi_\alpha)\}_{\alpha\in\Alpha}$. The natural map $\pi\co F\to T$ is then a continuous surjective 
map such that every fiber $F_t:=\pi^{-1}(t)$, for $t\in T$, with the induced topology, is a profinite $\Rho$-module isomorphic to 
$\lim_{\alpha\in\Alpha}(\cF_\alpha)_t$ (cf.\ Lemma~\ref{stalkspro}). 
Moreover, the $\Rho$-module structure on each fiber varies continuously with $\pi$, that is to say the natural maps:
\begin{equation}\label{modulestruct}
\mu\co\Rho\times F\to F\hspace{1cm}\mbox{and}\hspace{1cm}\sigma\co F\times_T F\to F,
\end{equation}
defined by $\mu(r,x):=r\cdot x$, for $r\in\Rho$ and $x\in F$, and by $\sg(x,y):=x+y$, for $(x,y)\in F\times_T F$, are continuous.
Note that the sheaf of continuous sections of $\pi\co F\to T$ is an object of the category $\Sh^\mathrm{top}_T(\Rho)$ which identifies with $\cF$.

Note, however, that, if $\{\cF_\alpha'\}_{\alpha\in\Alpha}$ is another cofiltered diagram in $\cS^\mathrm{fin}_T(\Rho)$ with the same limit 
$\cF\in\cS_T^\mathrm{pro}(\Rho)$, this does not imply that the associated topological space $(F',\pi')$ is isomorphic to $(F,\pi)$ over $T$, 
but only that $(F',\pi')$ and $(F,\pi)$ have isomorphic sheaves of continuous sections (both isomorphic to $\cF$). 
For this reason, we will need to give a more restrictive definition of the spaces we want to associate to sheaves
of profinite $\Rho$-modules.

Let $\cE_T^\mathrm{fin}(\Rho)$ be the category of \'etal\'e spaces associated to locally constant sheaves of finite discrete $\Rho$-modules. 
This is a subcategory of the category of finite \'etale coverings of the profinite space $T$. We then define:

\begin{definition}\label{proetalespace}A \emph{pro\'etal\'e space $(F,\pi)$ of profinite $\Rho$-modules over $T$} is the limit of a cofiltered
diagram $\{F_\alpha\}_{\alpha\in\Alpha}$ in $\cE_T^\mathrm{fin}(\Rho)$ with \emph{surjective} transition maps. 
We then let $\cE^\mathrm{pro}_T(\Rho)$ be the category with objects pro\'etal\'e spaces of profinite $\Rho$-modules over $T$ 
and with morphisms continuous maps over $T$ which restrict on the fibers over $T$ to homomorphisms of profinite $\Rho$-modules. 
\end{definition}

There is a natural functor $\Theta_T\co\cE^\mathrm{pro}_T(\Rho)\to\cS_T^\mathrm{pro}(\Rho)$ which assigns to a pro\'etal\'e space 
$(F,\pi)$ of profinite $\Rho$-modules over $T$ the sheaf of continuous sections of the map $\pi\co F\to T$.

\begin{theorem}\label{prosheavesproetale}The functor $\Theta_T\co\cE^\mathrm{pro}_T(\Rho)\to\cS_T^\mathrm{pro}(\Rho)$ 
is an equivalence of categories. In particular, $\cE^\mathrm{pro}_T(\Rho)$ is a complete abelian category.
\end{theorem}

\begin{proof}By definition of the category $\cE_T^\mathrm{fin}(\Rho)$, the restriction of the functor $\Theta_T$ to the subcategories 
$\cE_T^\mathrm{fin}(\Rho)$ and $\cS^\mathrm{fin}_T(\Rho)$ is an equivalence. 
By Lemma~\ref{compactsubquotient}, a sheaf $\cF\in\cS_T^\mathrm{pro}(\Rho)$ is the limit
of the cofiltered diagram (with surjective transition maps) of all its quotients which belong to $\cS^\mathrm{fin}_T(\Rho)$. 
Therefore, the functor $\Theta_T$ is at least split essentially surjective. The following lemma then shows that it is fully faithful 
and so an equivalence of categories:

\begin{lemma}\label{canonetale}For $\cF,\cF'\in\cS_T^\mathrm{pro}(\Rho)$, let $\{\cF_\alpha\}_{\alpha\in\Alpha}$ and
$\{\cF_\beta'\}_{\beta\in\Beta}$ be cofiltered diagrams in $\cS^\mathrm{fin}_T(\Rho)$ with surjective transition maps and limits 
$\cF$ and $\cF'$, respectively (cf.\ Lemma~\ref{compactsubquotient}) and let $\{(F_\alpha,\pi_\alpha)\}_{\alpha\in\Alpha}$ and 
$\{(F_\beta',\pi_\beta'\}_{\beta\in\Beta}$ be their respective associated cofiltered diagrams of \'etal\'e spaces in $\cE_T^\mathrm{fin}(\Rho)$ 
with limits $(F,\pi)$ and $(F',\pi')$ in $\cE^\mathrm{pro}_T(\Rho)$, respectively. Then, for any morphism of sheaves $\phi\co\cF\to\cF'$, 
there is a unique morphism $\tilde\phi\co F\to F'$ of pro\'etal\'e spaces of profinite $\Rho$-modules over $T$ which satisfies the following 
property: after identifying $\cF$ and $\cF'$ with the sheaves of sections of the maps $\pi\co F\to T$ and $\pi'\co F'\to T$, respectively,
the morphism $\phi$ identifies with the morphism induced by $\tilde\phi$ on sheaves of sections.
\end{lemma}

\begin{proof}By the first item of Lemma~\ref{locstronglyfiltered}, a morphism $\phi\co\cF\to\cF'$ is the limit of a cofiltered diagram
of morphisms $\{\phi_{\alpha\beta}\co\cF_\alpha\to\cF'_\beta\}_{\alpha\in\Alpha',\beta\in\Beta}$, where $\Alpha'$ is a cofiltered 
subdiagram of $\Alpha$. Then, we let $\tilde\phi\co F\to F'$ be the limit of the associated cofiltered diagram of morphisms in 
$\cE_T^\mathrm{fin}(\Rho)$. This is a morphism in $\cE^\mathrm{pro}_T(\Rho)$ which satisfies the property stated in the lemma.

In order to show that $\tilde\phi$ is uniquely determined by $\phi$ (and so conclude the proof of the lemma)
let us observe that $\tilde\phi$ is determined by its restriction to the fibers over $T$. 

By Lemma~\ref{stalkspro}, 
it is then enough to prove that, for every $t\in T$, the homomorphism of profinite $\Rho$-modules 
$\phi_t\co\lim_{\alpha\in\Alpha}(\cF_\alpha)_t\to\lim_{\beta\in\Beta}(\cF_\beta')_t$, induced by the cofiltered diagram
of morphisms $\{\phi_{\alpha\beta}\co\cF_\alpha\to\cF'_\beta\}_{\alpha\in\Alpha',\beta\in\Beta}$, is determined by $\phi$
and does not depend on the realizations of the sheaves $\cF$ and $\cF'$ as limits of the given cofiltered diagrams.

By Theorem~\ref{Pontryagindualitysheaves}, the morphism $\phi\co\cF\to\cF'$ determines a morphism $\phi^\vee\co(\cF')^\vee\to\cF^\vee$
in $\cS^\mathrm{dis}_T(\Rho)$ and, in particular, a morphism of stalks $\phi^\vee_t\co((\cF')^\vee)_t\to(\cF^\vee)_t$, for all $t\in T$.

Since the transition maps of the cofiltered diagrams$\{\cF_\alpha\}_{\alpha\in\Alpha}$ and $\{\cF_\beta'\}_{\beta\in\Beta}$ are surjective, 
by Remark~\ref{commlimits}, the sheaves $\cF^\vee$ and $(\cF')^\vee$ are the colimits of the filtered diagrams 
$\{\cF_\alpha^\vee\}_{\alpha\in\Alpha}$ and $\{(\cF_\beta')^\vee\}_{\beta\in\Beta}$.  
Since passing to stalks commutes with colimits, by Lemma~\ref{locallyconstanthom}, we have: 
\[(\cF^\vee)_t=\colim_{\alpha\in\Alpha}(\cF_\alpha^\vee)_t=\colim_{\alpha\in\Alpha}((\cF_\alpha)_t)^\vee\hspace{0.4cm}
\mbox{and}\hspace{0.4cm}((\cF')^\vee)_t=\colim_{\beta\in\Beta}((\cF_\beta')^\vee)_t=\colim_{\beta\in\Beta}((\cF_\beta')_t)^\vee.\]
Passing to Pontryagin duals, we then get the identities:
\[(\cF^\vee)_t^\vee=(\colim_{\alpha\in\Alpha}((\cF_\alpha)_t)^\vee)^\vee=\lim_{\alpha\in\Alpha}(\cF_\alpha)_t\hspace{0.4cm}
\mbox{and}\hspace{0.4cm}((\cF')^\vee)_t^\vee=(\colim_{\beta\in\Beta}((\cF_\beta')_t)^\vee)^\vee=\lim_{\beta\in\Beta}(\cF_\beta')_t.\]
So that the Pontryagin dual of $\phi^\vee_t$ induces a homomorphism of profinite $\Rho$-modules:
\[(\phi^\vee_t)^\vee\co\lim_{\alpha\in\Alpha}(\cF_\alpha)_t\to\lim_{\beta\in\Beta}(\cF_\beta')_t,\]
which clearly identifies with the homomorphism $\phi_t$ above, thus proving the unicity claim of the lemma.
\end{proof}
\end{proof}

\begin{remark}\label{stalkindependence}A corollary of the proof of Theorem~\ref{prosheavesproetale} is that, given cofiltered diagrams
$\{\cF_\alpha\}_{\alpha\in\Alpha}$ and $\{\cF_\beta'\}_{\beta\in\Beta}$ in $\cS^\mathrm{fin}_T(\Rho)$ with surjective transition maps and 
the same limit in $\cS^\mathrm{pro}_T(\Rho)$, there is a natural isomorphism 
$\lim_{\alpha\in\Alpha}(\cF_\alpha)_t\cong\lim_{\beta\in\Beta}(\cF_\beta')_t$, for all $t\in T$.
\end{remark}

\subsection{The category of pro\'etal\'e spaces of profinite $\Rho$-modules}
We let $\cE^\mathrm{pro}(\Rho)$ be the category whose objects are pro\'etal\'e spaces of profinite $\Rho$-modules over a profinite space 
and whose morphisms are defined as follows. For $(F,\pi)\in\cE^\mathrm{pro}_T(\Rho)$ and $(G,\pi')\in\cE^\mathrm{pro}_{T'}(\Rho)$,
a \emph{morphism $(\phi,\psi)\co(F,\pi)\to(G,\pi')$ of pro\'etal\'e spaces of profinite 
$\Rho$-modules} is a commutative diagram of continuous maps:
\begin{equation}\label{morphism}
\begin{array}{ccc}
F&\sr{\phi}{\to}&G\\
\;\;\da^{\pi}&&\;\;\da^{\pi'}\\
T&\sr{\psi}{\to}&T'
\end{array}
\end{equation}
such that, for all $t\in T$, the restriction $\phi_t\co F_t\to G_{\psi(t)}$ is a homomorphism of profinite $\Rho$-modules.
It is easy to check that to give such a morphism is equivalent to give a morphism between the corresponding sheaves of continuous 
sections $\phi\co\cF\to\psi^{-1}\cG$.

\begin{remark}\label{notcomplete}Even though, in virtue of Theorem~\ref{prosheavesproetale}, the category $\cE^\mathrm{pro}_T(\Rho)$
is complete, this is not true for the category $\cE^\mathrm{pro}(\Rho)$ (cf.\ Section~\ref{comparison}).
\end{remark}

\subsection{The category of \'etal\'e spaces of discrete $\Rho$-modules}\label{definitionetdis}
An \emph{\'etal\'e space $(F,\pi)$ of discrete $\Rho$-modules over a profinite space $T$} is the \'etal\'e space associated to a sheaf 
$\cF\in\cS^\mathrm{dis}_T(\Rho)$. By Proposition~\ref{locallycompact}, $F$ is a totally disconnected locally compact Hausdorff space.
The corresponding category $\cE^\mathrm{dis}_T(\Rho)$ is equivalent to the category $\cS^\mathrm{dis}_T(\Rho)$. 
In particular, it is a cocomplete abelian category.

The category $\cE^\mathrm{fin}_T(\Rho)$ also identifies with a full subcategory of $\cE^\mathrm{dis}_T(\Rho)$ and the latter
is the cocompletion of the former inside the category of \'etal\'e spaces on $T$, that is to say
every object of $\cE^\mathrm{dis}_T(\Rho)$ is a filtered colimit of a diagram in the category $\cE^\mathrm{fin}_T(\Rho)$.

As above, for $(G,\pi')$ the \'etal\'e space on a profinite space $T'$ associated to a sheaf $\cG\in\cS^\mathrm{dis}_{T'}(\Rho)$,
we define a \emph{morphism $(\phi,\psi)\co(F,\pi)\to(G,\pi')$} to be a commutative diagram of the type~\eqref{morphism}. 
Such a morphism is equivalent to a morphism of sheaves of sections $\phi\co\cF\to\psi^{-1}\cG$.
We denote by $\cE^\mathrm{dis}(\Rho)$ the corresponding category of \'etal\'e spaces of discrete $\Rho$-modules.

\subsection{Strongly filtered compactly generated weakly Hausdorff spaces}
\begin{definition}\label{stronglyfiltered}A filtered diagram $\{X_\alpha\}_{\alpha\in\Alpha}$ of inclusions of compact Hausdorff spaces is 
\emph{strongly filtered} if every compact subset of $X:=\colim_{\alpha\in\Alpha}X_\alpha$ lies in the image of some $X_\alpha$
(cf.\ \cite[Definition~3.4]{Strickland}). 
We then say that the topological space $X$ is \emph{strongly filtered} by the diagram $\{X_\alpha\}_{\alpha\in\Alpha}$. 
\end{definition}

By Corollary~\ref{completion}, every \'etal\'e space $(F,\pi)\in\cE^\mathrm{dis}_T(\Rho)$ is the colimit of a filtered diagram 
$\{(F_\alpha,\pi_\alpha)\}_{\alpha\in\Alpha}$ of inclusions in $\cE^\mathrm{fin}_T(\Rho)$. We have:

\begin{proposition}\label{stronglyfilt}With the above notations, $F$ is strongly filtered by the diagram $\{F_\alpha\}_{\alpha\in\Alpha}$. 
\end{proposition}

\begin{proof}By Lemma~\ref{lochomcomphaus}, if $K$ is a compact sub\'etal\'e space of $F$, its sheaf of sections $\cK$ is locally constant. 
In this case, the conclusion follows from item (ii) of Lemma~\ref{locstronglyfiltered}. 
The general case of a compact subset $C$ of $F$ can be reduced to this one by remarking that, 
by Proposition~\ref{locallycompact}, $C$ is contained in the union of a finite set $\{U_i\}_{i=1,\dots,k}$ of clopen compact subsets of $F$, such 
that each of the restrictions $\pi|_{U_i}\co U_i\to\pi(U_i)$ is a homeomorphism onto a clopen subset of $T$. 
The compact subspace $K:=\cup_{i=1}^kU_i$ of $F$ is then an \'etal\'e space over $T$ and we are reduced to the case already treated.
\end{proof}

\subsection{Pontryagin duality for \'etal\'e spaces}\label{Pontduality2}
For $(F,\pi)\in\cE^\mathrm{pro}_T(\Rho)$, with sheaf of sections $\cF$, we define the \emph{Pontryagin dual 
$(F^\vee,\pi^\vee)\in\cE^\mathrm{dis}_T(\Rho)$ of $(F,\pi)$} to be the \'etal\'e space associated to the Pontryagin dual $\cF^\vee$ 
of the sheaf $\cF$, defined in Section~\ref{Pontduality1}. To a morphism $\phi\co(F,\pi)\to(F',\pi')$ in $\cE^\mathrm{pro}_T(\Rho)$, 
we can then associate the morphism of \'etal\'e spaces $\phi^\vee\co((F')^\vee,(\pi')^\vee)\to(F^\vee,\pi^\vee)$ in $\cE^\mathrm{dis}_T(\Rho)$
determined by the Pontryagin dual of the morphism on sheaves of sections induced by $\phi$.
In this way, we have defined a \emph{Pontryagin duality functor}:
\[\Phi_T\co\cE^\mathrm{pro}_T(\Rho)\to\cE^\mathrm{dis}_T(\Rho)^\mathrm{op}.\]

\begin{theorem}\label{Pontryagindualityspaces}The Pontryagin duality functor establishes a duality between the abelian categories
$\cE^\mathrm{pro}_T(\Rho)$ and $\cE^\mathrm{dis}_T(\Rho)$ such that, for all $(F,\pi)\in\cE^\mathrm{pro}_T(\Rho)$ and $t\in T$,
there holds $(F^\vee)_t=(F_t)^\vee$ (so that we will denote both simply by $F^\vee_t$).
\end{theorem}

\begin{proof}From Theorem~\ref{Pontryagindualitysheaves} and Theorem~\ref{prosheavesproetale}, 
it follows that the functor $\Phi_T$ is an equivalence of categories. The identity $(F^\vee)_t=(F_t)^\vee$ 
instead follows from Lemma~\ref{stalkspro} and the last part of the proof of Theorem~\ref{prosheavesproetale}.
\end{proof}

The splitting of the functor $\Theta_T$ introduced in the proof of Theorem~\ref{prosheavesproetale}, in particular, determines an inverse
to the functor $\Phi_T$, that is to say a functor:
\[\Psi_T\co\cE^\mathrm{dis}_T(\Rho)\to\cE^\mathrm{pro}_T(\Rho)^\mathrm{op}\]
such that there are natural isomorphisms 
$\Psi_T^\mathrm{op}\circ\Phi_T\Ra\id_{\cE^\mathrm{pro}_T(\Rho)}$ and $\Phi_T^\mathrm{op}\circ\Psi_T\Ra\id_{\cE^\mathrm{dis}_T(\Rho)}$.

More explicitly, the \emph{Pontryagin duality functor} $\Psi_T$ assigns to $(F,\pi)\in\cE^\mathrm{dis}_T(\Rho)$
the limit $(F^\vee,\pi^\vee)\in\cE^\mathrm{pro}_T(\Rho)$ of the cofiltered diagram (with surjective transition maps) 
of the \'etal\'e spaces associated to the quotients in $\cS^\mathrm{fin}_T(\Rho)$ of the Pontryagin dual of the sheaf 
of continuous sections of $\pi\co F\to T$. By Lemma~\ref{stalkspro} and the last part of the proof of Theorem~\ref{prosheavesproetale}, 
we then also have that $(F^\vee)_t=(F_t)^\vee$, for all $t\in T$ (so that, as above, we will denote both simply by $F^\vee_t$).

\subsection{(Pro)\'etal\'e spaces of finite $\Rho$-modules}
Let $\cE^\mathrm{pro}_T(\Rho)_\mr{fin}$ and $\cE^\mathrm{dis}_T(\Rho)_\mr{fin}$ be, respectively, the full subcategories
of $\cE^\mathrm{pro}_T(\Rho)$ and $\cE^\mathrm{dis}_T(\Rho)$ with objects (pro)\'etal\'e spaces whose fibers are finite $\Rho$-modules. 
From the last part of the statement of Theorem~\ref{Pontryagindualityspaces}, it then follows that Pontryagin duality restricts to a duality 
between the subcategories $\cE^\mathrm{pro}_T(\Rho)_\mr{fin}$ and $\cE^\mathrm{dis}_T(\Rho)_\mr{fin}$. In fact, we have:

\begin{proposition}\label{finite}$\cE^\mathrm{pro}_T(\Rho)_\mr{fin}=\cE^\mathrm{dis}_T(\Rho)_\mr{fin}=\cE^\mathrm{fin}_T(\Rho)=
\cE^\mathrm{pro}_T(\Rho)\cap\cE^\mathrm{dis}_T(\Rho)$.
\end{proposition}

\begin{proof}The Pontryagin dual of a locally constant sheaf is also locally constant. Therefore it is enough to prove the identity
$\cE^\mathrm{dis}_T(\Rho)_\mr{fin}=\cE^\mathrm{fin}_T(\Rho)$. This is equivalent to the statement that, if an \'etal\'e space $(F,\pi)$ 
of discrete $\Rho$-modules has finite fibers, then the natural map $\pi\co F\to T$ is a covering map. By Lemma~\ref{lochomcomphaus},
it is enough to prove that $F$ is a compact Hausdorff space. By Proposition~\ref{locallycompact}, $F$ is Hausdorff and can be covered
by a finite union of compact subspaces, which implies the claim and so the proposition. 
\end{proof}

\section{Comparison with Melnikov, Ribes and Wilkes constructions}\label{comparison}
\subsection{'Sheaves' of profinite $\Rho$-modules}\label{coetale}
In \cite{Mel} and \cite[Chapter~V]{R} (cf.\ also \cite[Appendix~A]{Wilkes2}), a \emph{'sheaf' of profinite $\Rho$-modules} 
$\pi\co F\to T$ over a profinite space $T$ is defined to be a continuous surjective map of profinite spaces such that every fiber 
$F_t:=\pi^{-1}(t)$, for $t\in T$, with the induced topology, is a profinite $\Rho$-module and the maps~\eqref{modulestruct} 
are continuous. From \cite[Proposition~2.11]{Tang} and \cite[Proposition~1.11]{Haran}, it also follows that this definition is indeed equivalent 
to the definition of \'etale spaces given by Haran in \cite[Definition~1.4]{Haran}.

As it is remarked at the beginning of \cite[Section~3]{Wilkes}, these objects are no sheaves in any meaningful sense (for this reason,
following Wilkes' suggestion, above, we denoted them by 'sheaves' rather than \emph{sheaves}). In fact, by \cite[Theorem~3.3]{Wilkes}, 
the so called 'sheaves' of profinite $\Rho$-modules are the co\'etal\'e spaces associated to cosheaves of profinite $\Rho$-modules 
(cf.\ \cite[Definition~1.3]{Wilkes}) and, by \cite[Theorem~4.3]{Wilkes}, they are obtained as the Pontryagin duals of the \'etale 
spaces associated to sheaves (in the ordinary sense) of discrete $\Rho$-modules over profinite spaces.

More consistently with Wilkes' results than with his notations, we will then rather call these objects \emph{co\'etal\'e spaces of profinite 
$\Rho$-modules over a profinite space $T$} and denote by $\CoEt^\mathrm{pro}_T(\Rho)$ the corresponding category. From the remarks 
at the beginning of Section~\ref{definitionetpro} and Definition~\ref{proetalespace}, it follows that $\cE^\mathrm{pro}_T(\Rho)$ is a 
full subcategory of $\CoEt^\mathrm{pro}_T(\Rho)$.

Since $\cE^\mathrm{dis}_T(\Rho)$ is a proper (full) subcategory of the category of \'etale spaces associated to \emph{all} sheaves 
of discrete $\Rho$-modules over $T$ (cf.\ Section~\ref{skysheaves}), Theorem~\ref{Pontryagindualityspaces} and \cite[Theorem~4.3]{Wilkes} 
imply that $\cE^\mathrm{pro}_T(\Rho)$ is a proper subcategory of $\CoEt^\mathrm{pro}_T(\Rho)$ as well.

A more subtle characterization of the relation between the category $\cE^\mathrm{pro}_T(\Rho)$ and the category
$\CoEt^\mathrm{pro}_T(\Rho)$ follows from the lemma:

\begin{lemma}\label{tanglemma}For every $(F,\pi)\in\CoEt^\mathrm{pro}_T(\Rho)$, the sheaf of continuous sections
of the map $\pi\co F\to T$ belongs to the category $\cS^\mathrm{pro}_T(\Rho)$.
\end{lemma}

\begin{proof}By \cite[Proposition~2.11]{Tang}, for every $(F,\pi)\in\CoEt^\mathrm{pro}_T(\Rho)$, there exists a cofiltered diagram 
$\{(F_\alpha,\pi_\alpha,T_\alpha)\}_{\alpha\in\Alpha}$ of finite \'etal\'e spaces (i.e.\ both $T_\alpha$ and $F_\alpha$ are finite discrete sets, 
for all $\alpha\in\Alpha$) of finite $\Rho$-modules, such that $(F,\pi,T)\cong\lim_{\alpha\in\Alpha}(F_\alpha,\pi_\alpha,T_\alpha)$.
Let then $\cF$ (resp.\ $\cF_\alpha$, for $\alpha\in\Alpha$) be the sheaf of continuous sections of the map $\pi\co F\to T$
(resp.\ $\pi_\alpha\co F_\alpha\to T_\alpha$). 

Let $\psi_\alpha\co T\to T_\alpha$ be the natural map, for $\alpha\in\Alpha$. For every $\alpha\in\Alpha$, the pull-back sheaf 
$\psi_\alpha^{-1}(\cF_\alpha)$, which can also be described as the sheaf of sections of the map
$\pi_\alpha^\ast\co F_\alpha\times_{T_\alpha}T\to T$ obtained pulling back the map $\pi_\alpha$ along $\psi_\alpha$ 
(cf.\ \cite[Proposition~3.55]{Wedhorn}), is a locally constant sheaf of finite $\Rho$-modules. 

It is then clear that there is a natural morphism of sheaves $p_\alpha\co\cF\to\psi_\alpha^{-1}(\cF_\alpha)$, for every $\alpha\in\Alpha$,
and that every transition map $(F_\alpha,\pi_\alpha,T_\alpha)\to(F_\beta,\pi_\beta,T_\beta)$ induces a morphism of sheaves 
$\psi_\alpha^{-1}(\cF_\alpha)\to\psi_\beta^{-1}(\cF_\alpha)$ compatible with the pair of morphisms $p_\alpha,p_\beta$, so that
it is induced a natural isomorphism $\cF\cong\lim_{\alpha\in\Alpha}\psi_\alpha^{-1}(\cF_\alpha)$, where now, by definition, the sheaf
$\lim_{\alpha\in\Alpha}\psi_\alpha^{-1}(\cF_\alpha)$ belongs to the category $\cS^\mathrm{pro}_T(\Rho)$.
\end{proof}

Let $\Sh(\CoEt^\mathrm{pro}_T(\Rho))\subset\Sh^\mathrm{top}_T(\Rho)$ be the category of sheaves of continuous sections 
of co\'etal\'e spaces of profinite $\Rho$-modules over $T$. 
By Lemma~\ref{tanglemma}, we then have that $\Sh(\CoEt^\mathrm{pro}_T(\Rho))=\cS^\mathrm{pro}_T(\Rho)$. 
Hence, by Theorem~\ref{prosheavesproetale}, we conclude that:

\begin{theorem}\label{BoggivsWilkes}The natural embedding of the category 
$\cE^\mathrm{pro}_T(\Rho)$ in the category $\CoEt^\mathrm{pro}_T(\Rho)$
induces an equivalence between the category $\cE^\mathrm{pro}_T(\Rho)$
and the category $\Sh(\CoEt^\mathrm{pro}_T(\Rho))$.
\end{theorem}

\begin{remark}\label{objects}At the level of objects, the content of Theorem~\ref{BoggivsWilkes} can be explained as follows.
Let $\sim_\mathrm{shf}$ be the equivalence relation on the set of co\'etal\'e spaces of profinite $\Rho$-modules over $T$ such that two
co\'etal\'e spaces are shf-equivalent if and only if their associated sheaves of continuous sections are isomorphic. 
Then, Theorem~\ref{BoggivsWilkes} states that a shf-equivalence class contains a unique, up to isomorphism, pro\'etal\'e space of profinite 
$\Rho$-modules representative (cf.\ Remark~\ref{cosky} for a concrete example). 
\end{remark}

In particular, we have also defined a left inverse $\mathrm{R}_T\co\CoEt^\mathrm{pro}_T(\Rho)\to\cE^\mathrm{pro}_T(\Rho)$ to the
embedding functor $\Iota_T\co\cE^\mathrm{pro}_T(\Rho)\hookra\CoEt^\mathrm{pro}_T(\Rho)$. More explicitly, the functor $\mathrm{R}_T$
associates, to a co\'etal\'e space $(F,\pi)$ of profinite $\Rho$-modules over $T$, the pro\'etal\'e space of profinite $\Rho$-modules over $T$
associated to the sheaf of continuous sections of the map $\pi\co F\to T$ by the splitting of the functor $\Theta_T$ (cf.\ 
Theorem~\ref{prosheavesproetale}). In the next section, we will describe the functor $\mathrm{R}_T$ in terms of Pontryagin duality.

\subsection{Sheaves of discrete $\Rho$-modules}\label{alldiscrete}
Let $\Sh_T^\mathrm{dis}(\Rho)$ be the category of \emph{all} sheaves of torsion discrete $\Rho$-modules over a profinite space $T$ 
and let $\Et_T^\mathrm{dis}(\Rho)$ be the (equivalent) category with objects the associated \'etal\'e spaces. 
The categories $\cS_T^\mathrm{dis}(\Rho)$ and $\cE_T^\mathrm{dis}(\Rho)$ then identify with full subcategories
of $\Sh_T^\mathrm{dis}(\Rho)$ and $\Et_T^\mathrm{dis}(\Rho)$, respectively. Let us denote by
$\iota_T\co\cE^\mathrm{dis}_T(\Rho)\hookra\Et^\mathrm{dis}_T(\Rho)$ the natural embedding functor.
We will show that $\iota_T$ has a left inverse as well, although a more familiar one.

By definition, an \'etal\'e space $(F,\pi)\in\Et_T^\mathrm{dis}(\Rho)$ is locally homeomorphic to the profinite space $T$. 
Hence, if $F$ is Hausdorff, it is a totally disconnected, locally compact Hausdorff space and so, in particular, $(F,\pi)$ is 
the filtered colimit of its compact sub\'etal\'e subspaces. From Lemma~\ref{lochomcomphaus}, it then follows that 
$(F,\pi)\in\cE^\mathrm{dis}_T(\Rho)$. By Proposition~\ref{locallycompact}, the converse is also true. Therefore, we have:

\begin{lemma}\label{Hausdorff}The category $\cE^\mathrm{dis}_T(\Rho)$ is the full subcategory of $\Et^\mathrm{dis}_T(\Rho)$ whose
objects are Hausdorff spaces.
\end{lemma}

Let us recall (cf.\ \cite{Munster}) that, for every topological space $X$, there is a Hausdorff topological space $\Eta X$, called the
\emph{Hausdorff reflex of $X$}, and a quotient map $h_X\co X\to\Eta X$ such that every continuous map from $X$ to a 
Hausdorff space factors uniquely through $h_X$.

For $(F,\pi)\in\Et_T^\mathrm{dis}(\Rho)$, let $\Eta F$ be the Hausdorff reflex of the space $F$. By the universal property,
the continuous map $\pi\co F\to T$ induces a continuous map $\Eta(\pi)\co\Eta F\to T$.
Since continuous bijections of Hausdorff, totally disconnected, locally compact spaces are homeomorphisms, from
\cite[Theorem~6.1]{Munster}, it follows that $(\Eta F,\Eta(\pi))$, with the operations induced by $\Eta(\mu)$ and $\Eta(\sigma)$ 
(cf.~\eqref{modulestruct}), has a structure of \'etal\'e space of discrete $\Rho$-modules
and, by Lemma~\ref{Hausdorff}, we then have that $(\Eta F,\Eta(\pi))\in\cE_T^\mathrm{dis}(\Rho)$. 

We conclude that, for $(F,\pi)\in\Et_T^\mathrm{dis}(\Rho)$, the assignment $(F,\pi)\mapsto(\Eta F,\Eta(\pi))$
defines a functor (the \emph{Hausdorff reflex functor}) $\Eta_T\co\Et^\mathrm{dis}_T(\Rho)\to\cE^\mathrm{dis}_T(\Rho)$ which is a
left inverse and a left adjoint of the embedding functor $\iota_T\co\cE^\mathrm{dis}_T(\Rho)\hookra\Et^\mathrm{dis}_T(\Rho)$.

By the universal property of the Hausdorff reflex, it is also easy to check that, for $(F,\pi)\in\CoEt_T^\mathrm{pro}(\Rho)$,
the sheaf of continuous sections of the map $\pi\co F\to T$ is naturally isomorphic to the sheaf of continuous sections of
the map $\Eta(\pi^\vee)^\vee\co \Eta(F^\vee)^\vee\to T$, so that there is a natural isomorphism of functors:
 \[\mathrm{R_T}\Ra\Psi_T\circ\Eta_T^\mathrm{op}\circ\Phi_T,\]
where we denote by $\Phi_T\co\CoEt^\mathrm{pro}_T(\Rho)\to\Et^\mathrm{dis}_T(\Rho)^\mathrm{op}$ and 
$\Psi_T\co\Et^\mathrm{dis}_T(\Rho)\to\CoEt^\mathrm{pro}_T(\Rho)^\mathrm{op}$ Wilkes' Pontryagin duality functors
(cf.\ \cite[Theorem~4.3]{Wilkes}), that is to say, for all $(F,\pi)\in\CoEt^\mathrm{pro}_T(\Rho)$, there is a natural isomorphism
$\Eta(F^\vee,\pi^\vee)^\vee\cong\mathrm{R}_T(F,\pi)$. The functor $\mathrm{R}_T$ is then the right adjoint of the embedding of categories
$\Iota_T\co\cE^\mathrm{pro}_T(\Rho)\hookra\CoEt^\mathrm{pro}_T(\Rho)$.

Note that, for all $(F,\pi)\in\CoEt^\mathrm{pro}_T(\Rho)$, the Pontryagin dual of the natural quotient morphism 
$h_{F^\vee}\co (F^\vee,\pi^\vee)\to(\Eta(F^\vee),\Eta(\pi^\vee))$ defines a natural embedding: 
\[h_F^\vee\co\mathrm{R}_T(F,\pi)\hookra(F,\pi).\]
So that $\mathrm{R}_T(F,\pi)$ can also be characterized as the smallest co\'etal\'e subspace of $(F,\pi)$ which has the same sheaf
of continuous sections over $T$ as $(F,\pi)$. In conclusion, we have:

\begin{theorem}\label{BoggivsWilkes2}The category $\cE^\mathrm{dis}_T(\Rho)$ is the reflective subcategory of $\Et^\mathrm{dis}_T(\Rho)$ 
formed by Hausdorff \'etal\'e spaces and, under Pontryagin duality, $\cE^\mathrm{pro}_T(\Rho)$ is the corresponding coreflective subcategory 
of $\CoEt^\mathrm{pro}_T(\Rho)$.
\end{theorem}

\section{Natural constructions in $\cE^\mathrm{pro}_T(\Rho)$ and $\cE^\mathrm{dis}_T(\Rho)$}
\subsection{The sheaf Hom functor}
Neither of the categories $\cE^\mathrm{pro}(\Rho)$ and $\cE^\mathrm{dis}(\Rho)$ has an internal Hom functor. However, we can define 
an \'etal\'e space over $T$ of homomorphisms from an \'etal\'e space of profinite (resp.\ discrete) $\Rho$-modules to an \'etal\'e space
of discrete $\Rho$-modules on a profinite space $T$, provided, in the latter case, the target satisfies some technical conditions. 

For $(F,p)\in\cE^\mathrm{pro}_T(\Rho)$ and $(G,q)\in\cE_T^\mathrm{dis}(\Rho)$, let $\cF\in\cS^\mathrm{pro}_T(\Rho)$ and 
$\cG\in\cS_T^\mathrm{dis}(\Rho)$ be the respective sheaf of sections. The sheaf $\chom_{\ul{\Rho}}(\cF,\cG)$ is defined letting,
for an open subset $U$ of $T$:
\[\chom_{\ul{\Rho}}(\cF,\cG)(U):=\hom_\Rho(\cF|_{U},\cG|_{U})=\hom_\Rho(F|_{U},G|_{U}).\]

Let $(F,p)$ be the limit of a cofiltered diagram $\{(F_\alpha,\pi_\alpha)\}_{\alpha\in\Alpha}$ of surjective maps in 
$\cE^\mathrm{fin}_T(\Rho)$ and $(G,q)$ be the colimit of a filtered diagram $\{(G_\beta,q_\beta)\}_{\beta\in\Beta}$ of inclusions in 
$\cE^\mathrm{fin}_T(\Rho)$ (cf.\ Lemma~\ref{compactsubquotient} and Corollary~\ref{completion}). 
Let also $\cF_\alpha\in\cS_T^\mathrm{fin}(\Rho)$ 
be the sheaf of sections of $(F_\alpha,p_\alpha)$, for $\alpha\in\Alpha$, and $\cG_\beta\in\cS_T^\mathrm{fin}(\Rho)$ be the sheaf of sections of 
$(G_\beta,q_\beta)$, for $\beta\in\Beta$. By Proposition~\ref{stronglyfilt}, \cite[Lemma~3.8]{Strickland} and (i) of Lemma~\ref{locstronglyfiltered}, 
for a clopen subset $U$ of $T$, we have:
\[\begin{split}
\chom_{\ul{\Rho}}(\cF,\cG)(U)&=\hom_\Rho(F|_{U},G|_{U})\cong
\colim_{\beta\in\Beta}\colim_{\alpha\in\Alpha}\hom_\Rho(F_\alpha|_{U},G_\beta|_{U})=\\
&=\colim_{\beta\in\Beta}\colim_{\alpha\in\Alpha}\hom_\Rho(\cF_\alpha|_{U},\cG_\beta|_{U})
=\colim_{\beta\in\Beta}\colim_{\alpha\in\Alpha}\chom_{\ul{\Rho}}(\cF_\alpha,\cG_\beta)(U).
\end{split}\]
There is then a natural isomorphism of sheaves:
\[\chom_{\ul{\Rho}}(\cF,\cG)\cong\colim_{\beta\in\Beta}\colim_{\alpha\in\Alpha}\chom_{\ul{\Rho}}(\cF_\alpha,\cG_\beta).\]
By Lemma~\ref{locallyconstanthom}, we have that $\chom_{\ul{\Rho}}(\cF_\alpha,\cG_\beta)\in\cS_T^\mathrm{fin}(\ZZ)$
and then that $\chom_{\ul{\Rho}}(\cF,\cG)\in\cS_T^\mathrm{dis}(\ZZ)$.

\begin{definition}\label{homprodis}For $(F,p)\in\cE^\mathrm{pro}_T(\Rho)$ and $(G,q)\in\cE_T^\mathrm{dis}(\Rho)$, 
we let $(\chom_\Rho(F,G),\pi)\in\cE_T^\mathrm{dis}(\ZZ)$ be the \'etal\'e space associated to the sheaf $\chom_{\ul{\Rho}}(\cF,\cG)$.
\end{definition}

By Lemma~\ref{locallyconstanthom}, we have that $\chom_{\ul{\Rho}}(\cF_\alpha,\cG_\beta)_t=\hom_\Rho((\cF_\alpha)_t,(\cG_\beta)_t)$, 
for all $t\in T$, and, since the stalk functor commutes with filtered colimits, it follows that, for all $t\in T$:
\begin{equation}\label{stalkhom}
\chom_{\ul{\Rho}}(F,G)_t=\chom_{\ul{\Rho}}(\cF,\cG)_t=\hom_\Rho(\cF_t,\cG_t)=\hom_\Rho(F_t,G_t).
\end{equation}

In this way, we have defined the bifunctor:
\[\chom_\Rho(\_\,,\_)\co\cE^\mathrm{pro}_T(\Rho)^\mathrm{op}\times\cE^\mathrm{dis}_T(\Rho)\to\cE^\mathrm{dis}_T(\ZZ).\]

\begin{definition}\label{locfinite}We say that an \'etal\'e space $(F,\pi)\in\cE^\mathrm{dis}_T(\Rho)$ has \emph{locally finite torsion} 
if its sheaf of sections $\cF$ has the property that, for all $n\in\N^+$, the kernel of the morphism of sheaves $n\cdot\co\cF\to\cF$, induced
by multiplication by $n$, is a compact subsheaf of $\cF$. We denote by $\cE^\mathrm{dis}_T(\Rho)_\mathrm{lft}$ the corresponding 
full subcategory of $\cE^\mathrm{dis}_T(\Rho)$.
\end{definition}

For $(F,p)\in\cE^\mathrm{dis}_T(\Rho)$ and $(G,q)\in\cE_T^\mathrm{dis}(\Rho)_\mathrm{lft}$, let then $\cF\in\cS^\mathrm{dis}_T(\Rho)$ 
and $\cG\in\cS_T^\mathrm{dis}(\Rho)$ be the respective sheaves of sections. Then, $\cF$ is the colimit of a filtered diagram 
$\{\cF_\alpha\}_{\alpha\in\Alpha}$ in $\cS^\mathrm{fin}_T(\Rho)$ and there holds:
\[\chom_{\ul{\Rho}}(\cF,\cG)\cong\lim_{\alpha\in\Alpha}\chom_{\ul{\Rho}}(\cF_\alpha,\cG).\]

Since  $\cF_\alpha$ is locally constant, for all $\alpha\in\Alpha$, there is an $n_\alpha\in\N^+$ which annihilates $\cF_\alpha$.
Every homomorphism from $\cF_\alpha$ to $\cG$ then factors through the kernel of the morphism $n_\alpha\cdot\co\cG\to\cG$
so that $\chom_{\ul{\Rho}}(\cF_\alpha,\cG)=\chom_{\ul{\Rho}}(\cF_\alpha,\ker(n_\alpha\cdot))$. Now, $\ker(n_\alpha\cdot)$ is a compact subsheaf
of $\cG$ and hence, by Lemma~\ref{compactsubsheaf}, it is locally constant. By Lemma~\ref{locallyconstanthom},
this implies that $\chom_{\ul{\Rho}}(\cF_\alpha,\ker(n_\alpha\cdot))\in\cS^\mathrm{fin}_T(\ZZ)$, 
for all $\alpha\in\Alpha$, and then that:
\[\chom_{\ul{\Rho}}(\cF,\cG)\cong\lim_{\alpha\in\Alpha}\chom_{\ul{\Rho}}(\cF_\alpha,\ker(n_\alpha\cdot))\in\cS^\mathrm{pro}_T(\ZZ).\]

\begin{definition}\label{homdiscdisc}For $(F,p)\in\cE^\mathrm{dis}_T(\Rho)$ and $(G,q)\in\cE_T^\mathrm{dis}(\Rho)_\mathrm{lft}$,
let $(\chom_\Rho(F,G),\pi)\in\cE_T^\mathrm{pro}(\ZZ)$ be the \'etal\'e space associated to the sheaf $\chom_{\ul{\Rho}}(\cF,\cG)$
by the splitting of the functor $\Theta_T$ (cf.\ the proof of Theorem~\ref{prosheavesproetale}).
\end{definition}

\begin{remark}Note that, from Lemma~\ref{stalkspro}, Lemma~\ref{locallyconstanthom} and the above definition, it follows that, 
if the sheaf $\cG$ associated to $(G,q)\in\cE_T^\mathrm{dis}(\Rho)_\mathrm{lft}$ is an injective object in the category 
$\cS_T^\mathrm{dis}(\Rho)$, for all $(F,p)\in\cE^\mathrm{dis}_T(\Rho)$, there holds:
\begin{align*}\chom_\Rho(F,G)_t&=\lim_{\alpha\in\Alpha}\chom_\Rho(F_\alpha,G)_t=
\lim_{\alpha\in\Alpha}\chom_{\ul{\Rho}}(\cF_\alpha,\cG)_t=\\
&=\lim_{\alpha\in\Alpha}\hom_\Rho((\cF_\alpha)_t,\cG_t)=
\hom_\Rho(\cF_t,\cG_t)=\hom_\Rho(F_t,G_t).
\end{align*}
\end{remark}

In this way, we have defined a bifunctor:
\[\chom_\Rho(\_\,,\_)\co\cE^\mathrm{dis}_T(\Rho)^\mathrm{op}\times\cE^\mathrm{dis}_T(\Rho)_\mathrm{lft}\to\cE^\mathrm{pro}_T(\ZZ).\]

\begin{remark}\label{sections}For an \'etal\'e space $(F,\pi)$ of profinite or discrete $\Rho$-modules over $T$, let us denote by $F(T)$ 
the set of continuous sections of the natural map $\pi\co F\to T$. This has a natural structure of (topological) $\ZZ$-module.
For the \'etal\'e space $(\chom_\Rho(F,G),\pi)$ defined above, there is then a series of natural (continuous) isomorphisms:
\[\chom_\Rho(F,G)(T)\cong\chom_{\ul{\Rho}}(\cF,\cG)(T)\cong\hom_\Rho(\cF,\cG)\cong\hom_\Rho(F,G).\]
\end{remark}

\subsection{The completed Tensor functor}
For $(F,p)\in\cE^\mathrm{pro}_T(\Rho)$ and $(G,q)\in\cE^\mathrm{pro}_T(\ZZ)$, let us  define
a new \'etal\'e space $(F\wh{\otimes}G,\pi)\in\cE_T^\mathrm{pro}(\Rho)$ with the property that $(F\wh{\otimes}G)_t=F_t\wh{\otimes}G_t$, 
for all $t\in T$, where the right hand side is the completed tensor product of profinite modules (cf.\ for instance, \cite[Section~5.5]{RZ}).

Let $\cF\in\cS^\mathrm{pro}_T(\Rho)$ and $\cG\in\cS_T^\mathrm{pro}(\ZZ)$ be the sheaves of sections of $(F,p)$ and $(G,q)$.
Then, $\cF$ and $\cG$ are cofiltered limits of diagrams $\{\cF_\alpha\}_{\alpha\in\Alpha}$ and $\{\cG_\beta\}_{\beta\in\Beta}$ in 
$\cS^\mathrm{fin}_T(\Rho)$ and $\cS_T^\mathrm{fin}(\ZZ)$, respectively. For all $\alpha\in\Alpha$ and $\beta\in\Beta$, the sheaf 
$\cF_\alpha\otimes\cG_\beta$ is locally constant 
(cf.\ \cite[\href{https://stacks.math.columbia.edu/tag/093P}{Tag 093P}, Lemma~18.43.6]{stacks-project}). We then let:
\[\cF\wh{\otimes}\cG:=\lim_{\alpha\in\Alpha}\lim_{\beta\in\Beta}(\cF_\alpha\otimes\cG_\beta)\in\cS^\mathrm{pro}_T(\Rho).\]

\begin{definition}\label{tensprof}For $(F,p)\in\cE^\mathrm{pro}_T(\Rho)$ and $(G,q)\in\cE_T^\mathrm{pro}(\ZZ)$, let
$(F\wh{\otimes}G,\pi)\in\cE_T^\mathrm{pro}(\Rho)$ be the \'etal\'e space associated to the sheaf $\cF\wh{\otimes}\cG$
by the splitting of the functor $\Theta_T$ (cf.\ the proof of Theorem~\ref{prosheavesproetale}).
\end{definition}

It is easy to check that $(\cF_\alpha\wh{\otimes}\cG_\beta)_t=(\cF_\alpha\otimes\cG_\beta)_t=(\cF_\alpha)_t\otimes(\cG_\beta)_t$ and then 
$(F_\alpha\wh{\otimes} G_\beta)_t=(F_\alpha)_t\otimes(G_\beta)_t$, for all $t\in T$ and $\alpha\in\Alpha$, $\beta\in\Beta$. 
By Lemma~\ref{compactsubquotient}, we can also assume that the cofiltered diagrams $\{\cF_\alpha\}_{\alpha\in\Alpha}$ and 
$\{\cG_\beta\}_{\beta\in\Beta}$ have surjective transition maps. Since the tensor functor is right exact, this implies that the
cofiltered diagram $\{F_\alpha\wh{\otimes}G_\beta\}$ has  surjective transition maps. We then have that 
$F\wh{\otimes}G=\lim_{\alpha\in\Alpha}\lim_{\beta\in\Beta}F_\alpha\wh{\otimes}G_\beta$ and so, for all $t\in T$, there holds:
\begin{equation}\label{tensorfiber}
 (F\wh{\otimes}G)_t=\lim_{\alpha\in\Alpha}\lim_{\beta\in\Beta}(F_\alpha\wh{\otimes}G_\beta)_t=
 \lim_{\alpha\in\Alpha}\lim_{\beta\in\Beta}(F_\alpha)_t\otimes(G_\beta)_t=F_t\wh{\otimes}G_t.
 \end{equation}

In this way, we have defined a bifunctor with the desired properties:
\[(\_\,\wh{\otimes}\_)\co\cE^\mathrm{pro}_T(\Rho)\times\cE_T^\mathrm{pro}(\ZZ)\to\cE^\mathrm{pro}_T(\Rho).\]

\subsection{Tensor-Hom adjunction}
\begin{theorem}\label{adjunctiso}Let $(F,p)\in\cE^\mathrm{pro}_T(\Rho)$ and $(G,q)\in\cE_T^\mathrm{pro}(\ZZ)$ 
and $(L,\pi,T)\in\cE^\mathrm{dis}_T(\Rho)$.
\begin{enumerate}
\item There is a natural isomorphism of \'etal\'e spaces of discrete $\ZZ$-modules over $T$:
\[\chom_\Rho(F\wh{\otimes}G,L)\cong\chom_\ZZ(G,\chom_\Rho(F,L)).\]
\item There is a natural isomorphism of discrete $\ZZ$-modules:
\[\hom_\Rho(F\wh{\otimes}G,L)\cong\hom_\ZZ(G,\chom_\Rho(F,L)).\]
\end{enumerate}
\end{theorem}

\begin{proof}(i): Let $\cF\in\cS^\mathrm{pro}_T(\Rho)$, $\cG\in\cS_T^\mathrm{pro}(\ZZ)$ and $\cL\in\cS_T^\mathrm{dis}(\ZZ)$ be the sheaves 
of sections of $(F,p)$, $(G,q)$ and $(L,\pi)$, respectively. The claim of the theorem is then equivalent to the claim that there is a
natural isomorphism of sheaves of discrete $\ZZ$-modules over $T$:
\[\chom_{\ul{\Rho}}(\cF\wh{\otimes}\cG,\cL)\cong\chom_{\ul{\ZZ}}(\cG,\chom_{\ul{\Rho}}(\cF,\cL)).\]

For $U$ an open subset of $T$, the map: 
\[\Phi_U\co\hom_\Rho(\cF(U)\wh{\otimes}\cG(U),\cL(U))\to\hom_\ZZ(\cG(U),\chom_\Rho(\cF(U),\cL(U))),\]
which sends $f\co\cF(U)\wh{\otimes}\cG(U)\to\cL(U)$ to the map $\Phi_U(f)\co\cG(U)\to\chom_\Rho(\cF(U),\cL(U))$, 
defined by the assignment
$x\in\cG(U)\mapsto[y\in\cF(U)\mapsto f(x\wh{\scriptstyle\otimes}y)\in\cL(U)]$,
is an isomorphism by the standard adjunction (cf.\ \cite[(c) of Proposition~5.5.4]{RZ}).
The collection of maps $\{\Phi_U\}$ then determines an isomorphism of the corresponding presheaves
and so of the associated sheaves.
\smallskip

\noindent
(ii): This follows from item (i) and Remark~\ref{sections}.
\end{proof}

\subsection{Trivial \'etal\'e spaces}\label{trivial}
We associate to a profinite (resp.\ discrete) $\Rho$-module $M$ the \emph{trivial \'etal\'e space of $\Rho$-modules} 
$(M_T,p)$, where $M_T:=M\times T$ and $p$ is the projection onto the second factor. It is clear that this is an \'etal\'e
space of profinite (resp.\ discrete) $\Rho$-modules. This assignment defines exact functors:
 \[\times T\co\Rho\text{-}\mod^\mathrm{pro}\to\cE^\mathrm{pro}_T(\Rho)\hspace{0.4cm}\mbox{and}\hspace{0.4cm}
 \times T\co\Rho\text{-}\mod^\mathrm{dis}\to\cE^\mathrm{dis}_T(\Rho).\]
 The sheaf of continuous sections of the trivial \'etal\'e space $(M_T,p)$ is the constant sheaf $\ul{M}_T$ on $T$ associated to $M$.

\subsection{Cohomology of a profinite group with coefficients in a pro\'etal\'e space of torsion discrete modules}
Let $\Rho$ be one of the profinite rings $\ZZ$, $\Z_p$ or $\F_p$. In this section, we are going to define the cohomology of a profinite group 
$G$ with coefficients in an \'etal\'e space of discrete $\Rho[[G]]$-modules over a profinite set $T$. 

\begin{definition}\label{cohomology}Let $\{(F^\bullet,\pi^\bullet),d^\bullet\}$ be a cochain complex in the abelian category
$\cE^\mathrm{dis}_T(\Rho)$. We then denote by $(\cH^k(F^\bullet),\bar\pi^k)$ the $k$-th cohomology \'etal\'e space of this 
cochain complex which is indeed an \'etal\'e space of discrete $\Rho$-modules.
\end{definition}

For a profinite group $G$, let $\{\cB_i(G),\partial_i\}_{i\geq 0}$ be its associated \emph{homogenous bar resolution} (cf.\ \cite[Section~6.2]{RZ}).
This is a projective resolution of the trivial $\Rho[[G]]$-module $\Rho$ consisting of free profinite $\Rho$-modules.
In particular, for a profinite space $T$, the complex of trivial \'etal\'e spaces $\{(\cB_i(G)_T,p_i),(\partial_i)_T\}_{i\geq 0}$ over $T$ is a resolution
in the abelian category $\cE^\mathrm{pro}_T(\Rho[[G]])$ of the trivial \'etal\'e space $\Rho_T$ over $T$ (cf.\ Section~\ref{trivial}).
For $(F,p)\in\cE^\mathrm{dis}_T(\Rho[[G]])$, by applying the left exact functor
$\chom_{\Rho[[G]]}(\_,F)$ to this resolution, we then get a cochain complex: 
\[\{(\chom_{\Rho[[G]]}(\cB_\bt(G)_T,F),\pi^\bullet),(\partial_\bt)_T^\ast\}\]
in the category $\cE^\mathrm{dis}_T(\ZZ)$ and we define, for $k\geq 0$:
\[(\cH^k(G;F),\bar\pi^k_\ast):=(\cH^k(\chom_{\Rho[[G]]}(\cB_\bt(G)_T,F),\bar\pi^k)\in\cE^\mathrm{dis}_T(\ZZ).\]
By the isomorphism~\eqref{stalkhom}, for all $t\in T$, there is a natural isomorphism: 
\begin{equation}\label{stalkcohomology}
\cH^k(G;F)_t\cong H^k(G;F_t).
\end{equation}
In particular, $\cH^0(G;F)$ identifies with the fixed-point subspace $F^G$ of $F$.

\section{Products and coproducts of \'etal\'e spaces of modules}\label{directsumprod}
Coproducts of families of profinite groups indexed by profinite spaces were introduced in \cite{GR}, \cite{Haran} and \cite{Mel}.
A standard reference on the subject is \cite[Chaper~V]{R}. Here, we will first consider the notion of product of a family 
of discrete torsion modules indexed by a profinite set, which is simpler and more natural
from our point of view, and then obtain all the results about coproducts in  \cite{Haran} and \cite{Mel} by Pontryagin duality.

\subsection{The product functor}\label{defprodfunctor}
It will be useful to define products on the larger category of \'etal\'e spaces of all discrete $\Rho$-modules $\Et_T^\mathrm{dis}(\Rho)$
introduced in Section~\ref{alldiscrete}. 

The \emph{product functor}:
\[\prod_T\co\Et_T^\mathrm{dis}(\Rho)\to\Rho\text{-}\mod^\mathrm{dis}\]
is the right adjoint to the functor $\times T$, that is to say, it is characterized by the property that, for every \'etal\'e space in 
$\Et_T^\mathrm{dis}(\Rho)$ and every torsion discrete $\Rho$-module $A$, there holds:
\[\hom_{\Et_T^\mathrm{dis}(\Rho)}(A_T,F)\cong\hom_{\Rho\text{-}\mod^\mathrm{dis}}(A,\prod_TF).\]

The existence of the product functor $\prod_T$ follows from the Special Adjoint Functor Theorem
(cf.\ \cite[Corollary to Theorem~2, Section~8, Ch.~V]{MacLane} applied to the opposite categories). 
The category $\Rho\text{-}\mod^\mathrm{dis}$ is in fact small cocomplete, well-powered, with small hom-sets and a small 
generating set (e.g.\ a set of representatives for finite $\Rho$-modules),
while the category $\Et_T^\mathrm{dis}(\Rho)$ has small hom-sets and the functor $\times T$ preserve small colimits.

For an \'etal\'e space $(F,\pi)\in\Et_T^\mathrm{dis}(\Rho)$ the set of continuous sections $F(T)$ of the map $\pi\co F\to T$ has a natural
structure of torsion discrete $\Rho$-module.
We then have the following simple characterization of products of \'etal\'e spaces in $\Et_T^\mathrm{dis}(\Rho)$, which also prove the existence 
of products without resorting to  the Special Adjoint Functor Theorem:

\begin{proposition}\label{H^0}For all $(F,\pi)\in\Et_T^\mathrm{dis}(\Rho)$, there is a natural isomorphism of discrete $\Rho$-modules $\prod_T F\cong F(T)$.
\end{proposition}

\begin{proof}It is enough to show that the functor $H^0(T,\_)\co\Et_T^\mathrm{dis}(\Rho)\to\Rho\text{-}\mod^\mathrm{dis}$, defined by 
the assignment $F\mapsto F(T)$, is also a right adjoint to the functor $\times T$. Let $\cF\in\Sh_T^\mathrm{dis}(\Rho)$ be the sheaf of sections 
of the \'etal\'e space $(F,\pi)$, let $p\co T\to\{\ast\}$ be the map to the one-point space. For a torsion discrete $\Rho$-module $A$, 
the pullback sheaf $p^{-1}(\ul{A}_{\{\ast\}})$ is naturally isomorphic to the constant sheaf $\ul{A}_T$ on $T$ and, if $p_\ast(\cF)$ is 
the pushforward of a sheaf $\cF$ to the one-point space $\{\ast\}$, we have $F(T)=\cF(T)=p_\ast(\cF)(\{\ast\})$. 
By the adjunction isomorphism between pullback and pushforward of sheaves 
(cf.\ \cite[Proposition~3.49]{Wedhorn}), there is then a series of natural isomorphisms:

\[\begin{split}\hom_{\Et_T^\mathrm{dis}(\Rho)}(A_T,F)&\cong\hom_{\Sh_T(\Rho)}(\ul{A}_T,\cF)=\hom_{\Sh_T(\Rho)}(p^{-1}(\ul{A}_{\{\ast\}}),\cF)\cong\\
&\cong\hom_{\Sh_{\{\ast\}}(\Rho)}(\ul{A}_{\{\ast\}},p_\ast(\cF))=\hom_{\Rho\text{-}\mod^\mathrm{dis}}(A,F(T)).
\end{split}\]
\end{proof}

\begin{remark}\label{finiteprod}In particular, from Proposition~\ref{H^0}, it follows that, for $T$ a finite discrete set and $(F,\pi)\in\Et_T^\mathrm{dis}(\Rho)$, 
the product $\prod_T F$ coincides with the ordinary direct product of the $\Rho$-modules in the fibers of $(F,\pi)$.
\end{remark}

An important property of sheaves on profinite spaces is that they are all soft (cf.\ \cite[Lemma~3.3]{Pierce}). Since profinite spaces are compact
and Hausdorff, these sheaves are also acyclic (cf.\ \cite[Theorem~9.14]{Wedhorn}). From Proposition~\ref{H^0}, it then immediately follows:

\begin{corollary}\label{prodexact}The functor $\prod_T\co\Et_T^\mathrm{dis}(\Rho)\to\Rho\text{-}\mod^\mathrm{dis}$ is exact.
\end{corollary}

For a profinite space $T$ and a torsion discrete $\Rho$-module $A$, let $A^T:=\cC(T,A)$ be the set of continuous maps from $T$ to $A$
with the obvious torsion discrete $\Rho$-module structure, which we call the \emph{free product of $T$ copies of $A$}.
An immediate consequence of Proposition~\ref{H^0} is then the following explicit description of the product $\prod_T A_T$ 
of the trivial \'etal\'e space over $T$ with coefficients in a torsion discrete $\Rho$-module $A$:

\begin{corollary}\label{freeproddescr}There is a natural isomorphism of discrete $\Rho$-modules:
\[\prod_T A_T\cong A^T.\]
\end{corollary}

\begin{proof}The set of continuous sections of the projection $p_2\co A\times T\to T$ clearly identifies with the set of continuous maps 
from $T$ to $A$.
\end{proof}

Another interesting consequence of Proposition~\ref{H^0} is:

\begin{corollary}\label{proddecomposition}Let $(F,\pi)\in\Et_T^\mathrm{dis}(\Rho)$ and let $T=\coprod_{i=1}^kU_i$ be a decomposition 
into disjoint clopen subsets. There is then a natural isomorphism:
\[\prod_T F=\prod_{i=1}^kF(U_i)=\prod_{i=1}^k\prod_{U_i}F|_{U_i}.\]
\end{corollary}

 \subsection{Skyscraper \'etal\'e spaces}\label{skysheaves}
For $A\in\Rho\text{-}\mod^\mathrm{dis}$ and a point $x\in T$, let $skysc_x(A)\in\Sh_T^\mathrm{dis}(\Rho)$ be the skyscraper sheaf supported 
at $x$. Let us recall that, if $i\co\{x\}\hookra T$ is the natural inclusion and $A_{\{x\}}$ is the constant sheaf on $\{x\}$, we have 
$skysc_x(A):=i_\ast(A_{\{x\}})$. The \emph{skyscraper \'etal\'e space $(\sky_x(A),p_x)$ supported at $x$} is then the \'etal\'e space in 
$\Et_T^\mathrm{dis}(\Rho)$ associated to the skyscraper sheaf $skysc_x(A)$. This assignment defines an exact functor:
 \[\sky_x\co\Rho\text{-}\mod^\mathrm{dis}\to\Et^\mathrm{dis}_T(\Rho).\]
 
 Note that, in general, by Proposition~\ref{locallycompact}, $(\sky_x(A),p_x)\notin\cE_T^\mathrm{dis}(\Rho)$, 
 since, unless $x$ is an isolated point of $T$ or the $\Rho$-module $A$ is trivial, the space $\sky_x(A)$ is not Hausdorff. 
 
 For a sheaf $\cF\in\Sh_T^\mathrm{dis}(\Rho)$ and a point $x\in T$, there is a natural morphism
 $sk_x\co\cF\to skysc_x(\cF_x)$. Let us denote by $\mathrm{sk}_x\co F\to \sky_x(F_x)$ the corresponding map of \'etal\'e spaces. 
 This map then induces the adjunction isomorphism, for $A\in\Rho\text{-}\mod^\mathrm{dis}$, 
 $(F,p)\in\Et^\mathrm{dis}(\Rho)$ and $x\in T$, between the fiber functor $F\mapsto F_x$ and the functor $A\mapsto\sky_x(A)$:
\[ \hom_\Rho(F_x,A)\cong\hom_{\Et_T^\mathrm{dis}(\Rho)}(F, \sky_x(A)).\]
 
 Let $\{M_t\}_{t\in S}$ be a family of torsion discrete $\Rho$-modules indexed by a finite subset $S$ of the profinite space $T$.
 We can then associate to this family the sheaf $\prod_{t\in S}skysc_t(M_t)\in\Sh_T^\mathrm{dis}(\Rho)$ and to this sheaf
the  \'etal\'e space $(\prod_{t\in S}\sky_t(M_t),p_S)\in\Et_T^\mathrm{dis}(\Rho)$. For a sheaf $\cF\in\Sh_T^\mathrm{dis}(\Rho)$,
there is a natural morphism $\mathrm{sk}_S\co F\to\prod_{t\in S}\sky_t(F_t)$ and then an isomorphism:
 \begin{equation}\label{adjskysum}
\prod_{t\in S} \hom_\Rho(F_t,M_t)\cong\hom_{\Et_T^\mathrm{dis}(\Rho)}(F,\prod_{t\in S}\sky_t(M_t)).
 \end{equation}

\begin{lemma}\label{skyprod}For $\{M_t\}_{t\in S}$ a family of torsion discrete $\Rho$-modules indexed by a finite subset $S$, 
there is a natural isomorphism $\prod_T(\prod_{t\in S}\sky_t(M_t))\cong\prod_{t\in S} M_t$.
\end{lemma}

\begin{proof} By the isomorphism~\eqref{adjskysum}, for every torsion discrete $\Rho$-module $A$, we have:
\[\hom_{\Et_T^\mathrm{dis}(\Rho)}(A_T,\prod_{t\in S}\sky_t(M_t))\cong\prod_{t\in S}\hom_{\Rho\text{-}\mod^\mathrm{dis}}(A,M_t)
\cong\hom_{\Rho\text{-}\mod^\mathrm{dis}}(A,\prod_{t\in S}M_t),\]
which implies the lemma.
\end{proof}

\begin{remark}\label{cosky}For $A\in\Rho\text{-}\mod^\mathrm{pro}$ and a point $x\in T$, let the \emph{coskyscraper co\'etal\'e space}
$(\cosky_x(A),q_x)\in\CoEt_T^\mathrm{pro}(\Rho)$ be the Pontryagin dual (in the sense of \cite[Theorem~4.3]{Wilkes})
of the skyscraper \'etal\'e space $(\sky_x(A^\vee),p_x)\in\CoEt_T^\mathrm{pro}(\Rho)$ constructed above.
It is then easy to see that $\cosky_x(A)$ is obtained from the disjoint union of $T$ and $A$ by identifying the point 
$x\in T$ with the point $\{0\}\in A$. 

Note that, if $x\in T$ is not an isolated point, the sheaf of continuous sections of the natural
map $q_x\co\cosky_x(A)\to T$ only contains $0$-sections, so that $\mathrm{R}_T(\cosky_x(A),q_x)$ (cf.\ Section~\ref{coetale}) 
is the trivial pro\'etal\'e space $q_x\co\{0\}\times T\to T$. Consistently with the results of Section~\ref{alldiscrete},
there also holds $\Eta\sky_x(A^\vee)=\{0\}\times T$, so that we have indeed that $\mathrm{R}_T(\cosky_x(A))\cong\Eta\sky_x(A^\vee)^\vee$.
\end{remark}

\subsection{Properties of the product functor}
The canonical morphism $\omicron\co(\prod_TF)_T\to F$,  determined by the adjunction counit, induces, for all $t\in T$,
a canonical homomorphism of discrete $\Rho$-modules $\omicron_t\co\prod_TF\to F_t$. 

\begin{proposition}\label{prodprop}For $(F,p)\in\Et_T^\mathrm{dis}(\Rho)$, there holds:
\begin{enumerate}
\item For any finite subset $S$ of $T$, the natural homomorphism
$\prod_{t\in S}\omicron_t\co\prod_TF\to\prod_{t\in S} F_t$ is surjective.
\item There holds $\bigcap_{t\in T}\ker \omicron_t=\{0\}$, that is to say,
an element $x\in\prod_T F$ is uniquely determined by the family of elements $(\omicron_t(x))_{t\in T}$.
\end{enumerate}
\end{proposition}

\begin{proof}(i): By an argument similar to the one in the proof of Proposition~\ref{H^0}, we see that $\prod_{t\in S}\omicron_t=\prod_T(\mathrm{sk}_S)$.
The conclusion then follows from Corollary~\ref{prodexact} and Lemma~\ref{skyprod}.
\medskip

\noindent
(ii): By Proposition~\ref{H^0}, this item of the proposition is equivalent to 
the fact that a global section $x\in F(T)=\prod_T F$ is uniquely determined by the family of germs $(x_t)_{t\in T}$.
\end{proof}

\begin{definition}We will also denote the product $\prod_T F$ by $\prod_{t\in T}^\mathrm{top}F_t$ and call it the 
\emph{direct product of the \'etal\'e space of torsion discrete $\Rho$-modules $(F,p)$ over $T$}. As we observed above, for $T$ 
discrete (and then finite), there holds $\prod_{t\in T}^\mathrm{top}F_t=\prod_{t\in T}F_t$.
\end{definition}

\subsection{Locally constant sections of products of \'etal\'e spaces in $\cE^\mathrm{dis}_T(\Rho)$}
An important property of the standard direct product $\prod_{i\in I}M_i$ of a family $\{M_i\}_{i\in I}$ of $\Rho$-modules is that 
each member $M_j$ of the family can be uniquely identified with a submodule of $\prod_{i\in I}M_i$ with the property that the restriction
of the natural projection $p_j\co\prod_{i\in I}M_i\to M_j$ to $M_j$ is the identity while, for $j\neq k\in I$, the restriction of the 
projection $p_k\co\prod_{i\in I}M_i\to M_k$ to $M_j$ is the zero map. In general, this property does not hold
for the direct product of an \'etal\'e space of torsion discrete $\Rho$-modules as defined above. 
But a similar (much weaker) property holds, if we restrict to the subcategory $\cE^\mathrm{dis}_T(\Rho)\subset\Et_T^\mathrm{dis}(\Rho)$.

By Proposition~\ref{locallycompact}, for $(F,p)\in\cE^\mathrm{dis}_T(\Rho)$ and  an element $x\in F$, there is a fundamental system 
$\{U_\alpha\}_{\alpha\in\Alpha}$ of clopen neighborhoods of $x$ in $F$ such that $\{p(U_\alpha)\}_{\alpha\in\Alpha}$ is a fundamental system 
of clopen neighborhoods of $p(x)$ in $T$ and the restriction $p|_{U_\alpha}\co U_\alpha\to p(U_\alpha)$
is a homeomorphism. We then have:

\begin{proposition}\label{separate}With the above notations, for all $\alpha\in\Alpha$, 
there is a global section $\td x_\alpha\in\prod_T F=F(T)$ with the property that the restriction of $\td x_\alpha$ to $p(U_\alpha)$ 
is equal to $p|_{U_\alpha}^{-1}$ and the restriction of $\td x_\alpha$ to $T\ssm p(U_\alpha)$ is the zero section.
\end{proposition}

\subsection{The coproduct functor}
In this section, we introduce the coproduct of an \'etal\'e space of profinite $\Rho$-modules. 
The \emph{coproduct functor}:
\[\bigoplus_T\co\cE_T^\mathrm{pro}(\Rho)\to\Rho\text{-}\mod^\mathrm{pro}\]
is defined to be the left adjoint of the functor $\times T$ (cf.\ Section~\ref{trivial}), that is to say, is characterized by the property that, 
for every \'etal\'e space of profinite $\Rho$-modules $F$ over $T$ and every profinite $\Rho$-module $A$, there is the identity:
\[\hom_{\cE_T^\mathrm{pro}(\Rho)}(F, A_T)=\hom_{\Rho\text{-}\mod^\mathrm{pro}}(\bigoplus_TF,A).\]

The existence of the coproduct functor $\bigoplus_T$ immediately follows from a form of the Special Adjoint Functor Theorem
(cf.\ \cite[Corollary to Theorem~2, Section~8, Ch.~V]{MacLane}). The category $\Rho\text{-}\mod^\mathrm{pro}$ is in fact
small complete, well-powered, with small hom-sets and a small cogenerating set (e.g.\ a set of representatives for finite $\Rho$-modules),
while the category $\cE_T^\mathrm{pro}(\Rho)$ has small hom-sets and the functor $\times T$ preserve small limits.

\begin{definition}We will also denote the coproduct $\bigoplus_T F$ by the more suggestive notation $\bigoplus_{t\in T}^\mathrm{top}F_t$ 
and call it the \emph{direct sum of the \'etal\'e space of profinite $\Rho$-modules $(F,p)$ over $T$}. For $T$ discrete (and then finite),
there holds $\bigoplus_{t\in T}^\mathrm{top}F_t=\bigoplus_{t\in T}F_t$
\end{definition}

\begin{remark}\label{RibesDef}
Note that the datum of a morphism $F\to A_T$ of \'etal\'e spaces of profinite $\Rho$-modules over $T$ is equivalent to  the datum a 
\emph{morphism} $F\to A$ as defined in \cite[Section~5.1]{R}, so that the universal property which defines the coproduct functor above
is equivalent to the one which defines free products of 'sheaves' of profinite groups in \cite{R}.
\end{remark}

\subsection{Products, coproducts and Pontryagin duality}
For $A\in\Rho\text{-}\mod^\mathrm{pro}$ or $\Rho\text{-}\mod^\mathrm{dis}$, there is clearly a natural isomorphism
$(A_T)^\vee\cong A^\vee_T$. By Theorem~\ref{Pontryagindualityspaces} and the universal properties of direct sums and 
direct products of \'etal\'e spaces of $\Rho$-modules over $ T$, we then have:

\begin{proposition}\label{dualproduct}\leavevmode\begin{enumerate}
\item For $(F,p)\in\cE_T^\mathrm{pro}(\Rho)$, let $(F^\vee,p^\vee)\in\cE_T^\mathrm{dis}(\Rho)$ be its Pontryagin dual. 
There is then a natural isomorphism of discrete $\Rho$-modules: 
\[(\bigoplus_T F)^\vee\cong\prod_T F^\vee.\]
\item For $(F,p)\in\cE_T^\mathrm{dis}(\Rho)$, let $(F^\vee,p^\vee)\in\cE_T^\mathrm{pro}(\Rho)$ be its Pontryagin dual. 
There is then a natural isomorphism of profinite $\Rho$-modules:
\[(\prod_T F)^\vee\cong\bigoplus_T F^\vee.\]
\end{enumerate}
\end{proposition}

From Proposition~\ref{dualproduct}, Corollary~\ref{prodexact} and Corollary~\ref{proddecomposition}, it immediately follows:

\begin{corollary}\label{sumexact}The functor $\bigoplus_T$ is exact.
\end{corollary}

\begin{corollary}\label{coproddecomposition}Let $(F,\pi)\in\cE_T^\mathrm{pro}(\Rho)$ and let $T=\coprod_{i=1}^kU_i$ 
be a decomposition into disjoint clopen subsets. There is then a natural isomorphism:
\[\bigoplus_T F=\bigoplus_{i=1}^k\bigoplus_{U_i}F|_{U_i}.\]
\end{corollary}

\subsection{Properties of the coproduct functor}
Let $\omega\co F\to(\bigoplus_TF)_T$ be the canonical morphism determined by the adjunction unit.  
This homomorphism induces on the fibers the
canonical homomorphisms of profinite $\Rho$-modules $\omega_t\co F_t\to\bigoplus_TF$ and it is clear that 
$\omega_t=\omicron_t^\vee$, for $t\in T$. 

In particular, the family of homomorphisms $\{\omega_t\}_{t\in T}$ induces a
natural homomorphism of $\Rho$-modules from the standard direct sum of the family $\{F_t\}_{t\in T}$:
\[\Omega\co\bigoplus_{t\in T}F_t\to\bigoplus_TF.\]
The following proposition then shows that $\bigoplus_TF$ is a profinite completion of $\bigoplus_{t\in T}F_t$:

\begin{proposition}\label{coprodprop}For $(F,p)$ an \'etal\'e space of profinite $\Rho$-modules over $T$, the natural
homomorphism of $\Rho$-modules $\Omega\co\bigoplus_{t\in T}F_t\to\bigoplus_TF$ is injective and has dense image.
\end{proposition}

\begin{proof}To prove that $\Omega$ is injective, it is enough to prove that, for any finite subset $S$ of $T$, the natural homomorphism
$\bigoplus_{t\in S}\omega_t\co\bigoplus_{t\in S}F_t\to\bigoplus_TF$ is injective. But this immediately follows from (i) of Proposition~\ref{prodprop}
and Proposition~\ref{dualproduct}.

For the last statement of the proposition, it is enough to show that $\bigoplus_T F$ is topologically generated by the family of submodules $\{F_t\}_{t\in T}$.
If this were not the case, then we could factor the morphism $\omega$ through a morphism $\omega'\co F\to B_T$, for some proper 
profinite submodule $B$ of $\bigoplus_T F$. But then the identity map on $\bigoplus_T F$ would admit a similar factorization, which is absurd.
\end{proof}


\subsection{Compatibility with Hom and Tensor functors}
In this section, we will show that the functors $\bigoplus_T$ and $\prod_T$ satisfy the standard compatibility properties 
with respect to the Hom and Tensor functors. More precisely, we have:

\begin{theorem}\label{compatibility}For $(F,p)\in\cE_T^\mathrm{pro}(\Rho)$, $(G,q)\in\cE_T^\mathrm{dis}(\Rho)$, 
$A\in\Rho\text{-}\mod^\mathrm{pro}$, $B\in\ZZ\text{-}\mod^\mathrm{pro}$ and $C\in\Rho\text{-}\mod^\mathrm{dis}$, 
there are natural isomorphisms:
\begin{subequations}
\begin{equation}\tag{i}\hom_\Rho(A,\prod_T G)\cong\prod_T\chom_\Rho(A_T,G);\end{equation}
\begin{equation}\tag{ii}\hom_\Rho(\bigoplus_TF,C)\cong\prod_T\chom_\Rho(F,C_T).\end{equation}
\begin{equation}\tag{iii}\bigoplus_T (F\wh{\otimes}B_T)\cong(\bigoplus_TF)\wh{\otimes}B;\end{equation}
\end{subequations}
\end{theorem}

\begin{proof}(i): By Proposition~\ref{H^0}, Remark~\ref{sections} and the definition of the functor $\prod_T$,
there is a series of natural isomorphisms:
\[\prod_T\chom_\Rho(A_T,G)\cong\chom_\Rho(A_T,G)(T)=\hom_\Rho(A_T,G)=\hom_\Rho(A,\prod_T G).\]
\smallskip

\noindent
(ii): By Pontryagin duality, item (i) of Proposition~\ref{dualproduct} and the previous item, there is a series of natural isomorphisms:
\[\hom_\Rho(\bigoplus_TF,C)\cong\hom_\Rho(C^\vee,\prod_TF^\vee)\cong\prod_T\chom_\Rho(C^\vee_T,F^\vee).\]
Now, from the definition of the hom \'etal\'e space $\chom_\Rho(F,C_T)$, it immediately follows that there is a natural isomorphism
of \'etal\'e spaces of torsion discrete $\ZZ$-modules:
\[\chom_\Rho(F,C_T)\cong\chom_\Rho(C^\vee_T,F^\vee),\]
from which the conclusion follows.
\smallskip

\noindent
(iii): There is a canonical homomorphism of \'etal\'e spaces $\omega\co F\to(\bigoplus_TF)_T$ 
(cf.\ the proof of Proposition~\ref{coprodprop}). Tensoring by the constant \'etal\'e space $B_T$, we get the natural homomorphism of \'etal\'e spaces
$\omega\wh{\otimes}B_T\co F\wh{\otimes}B_T\to(\bigoplus_TF)\wh{\otimes}B_T=((\bigoplus_TF)\wh{\otimes}B)_T$ which, by the universal property 
of the functor $\bigoplus_T$, induces the natural homomorphism of profinite $\Rho$-modules:
\[\bigoplus_T(F\wh{\otimes}B_T)\to(\bigoplus_TF)\wh{\otimes}B.\]
In order to prove that this is an isomorphism, it is enough to show that, for every finite $R$-module $D$ (cf.\ the hypotheses on the $\ZZ$-algebra $\Rho$
at the beginning of Section~\ref{algdef}), it induces a natural isomorphism:
\[\hom_\ZZ((\bigoplus_TF)\wh{\otimes}B,D)\cong\hom_\ZZ(\bigoplus_T(F\wh{\otimes}B_T),D).\]
This follows from the series of isomorphisms (cf.\ \cite[(c) of Proposition~5.5.4]{RZ}, the adjunction identity which defines the functor $\bigoplus_T$ 
and (ii) of Theorem~\ref{adjunctiso}):
\[\begin{array}{l}
\hom_\ZZ((\bigoplus_T F)\wh{\otimes}B,D)\cong\hom_\ZZ(\bigoplus_TF,\hom_\ZZ(B,D))\cong\hom_\ZZ(F,\hom_\ZZ(B,D)_T)\cong\\
\\
\cong\hom_\ZZ(F,\chom_\ZZ(B_T,D_T))\cong\hom_\ZZ(F\wh{\otimes}B_T,D_T)\cong\hom_\ZZ(\bigoplus_T(F\wh{\otimes}B_T),D).
\end{array}\]
\end{proof}

\subsection{Free sums of profinite $\Rho$-modules}
In this section, we study in more detail the coproduct of a constant \'etal\'e space over $T$. This can be related to a standard construction
in the theory of profinite $\Rho$-modules. Let $T$ be a profinite space obtained as the cofiltered limit of a diagram of finite discrete sets 
$\{T_\alpha\}_{\alpha\in\Lambda}$. For $A\in\Rho\text{-}\mod^\mathrm{pro}$, let $A[T_\alpha]:=\bigoplus_{t\in T_\alpha} A$ be the
direct sum of $|T_\alpha|$ copies of $A$ indexed by $T_\alpha$, for $\alpha\in\Lambda$. A map of finite sets $T_\alpha\to T_\beta$
induces a homomorphism of $\Rho$-modules $A[T_\alpha]\to A[T_\beta]$, for $\alpha,\beta\in\Lambda$, so that
$\{A[ T_\alpha]\}_{\alpha\in\Lambda}$ is a cofiltered diagram of profinite $\Rho$-modules. 
The \emph{free sum of $A$ over $T$} is then defined to be the limit: 
\[A[[T]]=\lim_{\alpha\in\Lambda}A[T_\alpha].\]

\begin{proposition}\label{freesum}For $A\in\Rho\text{-}\mod^\mathrm{pro}$ and $T$ a profinite space, there is a natural isomorphism
of profinite $\Rho$-modules $\bigoplus_T A_T\cong A[[T]]$.
\end{proposition}

\begin{proof}By Proposition~\ref{dualproduct} and Corollary~\ref{freeproddescr}, it is enough to show
that there is a natural isomorphism $((A^\vee)^T)^\vee\cong A[[T]]$. Let, as above, $T$ be the cofiltered limit 
of a diagram of finite discrete sets $\{T_\alpha\}_{\alpha\in\Lambda}$. We then have:
\[((A^\vee)^T)^\vee\cong((A^\vee)^{\lim_{\alpha\in\Lambda}T_\alpha})^\vee\cong(\colim_{\alpha\in\Lambda}(A^\vee)^{T_\alpha})^\vee
\cong\lim_{\alpha\in\Lambda}((A^\vee)^{T_\alpha})^\vee\cong\lim_{\alpha\in\Lambda}\bigoplus_{T_\alpha} A_{T_\alpha}=A[[T]].\]
\end{proof}

\subsection{Relative free products and sums}
Let  $\pi\co T\to S$ be a continuous surjective map of profinite spaces. Then, the inverse image $\pi^{-1}(s)$ 
is a profinite set and, for a given torsion discrete (resp.\ profinite) $\Rho$-module $A$, we get a family of free products $\{A^{\pi^{-1}(s)}\}_{s\in S}$
(resp.\ of free sums $\{A[[\pi^{-1}(s)]]\}_{s\in S}$). It is not difficult to check directly that  these families of closed submodules of $A^T$ and $A[[T]]$, 
respectively, have a natural structure of \'etal\'e spaces over $S$. However, here, we give a more direct and simpler proof of this fact.

For an \'etal\'e space $F$ over a space $T$ with sheaves of sections $\cF$, given a continuous map $\pi\co T\to S$, 
let $\pi_\ast F$ be the \'etal\'e space over $S$ associated to the sheaf $\pi_\ast\cF$ on $S$.

\begin{definition}\label{relative}Let $\pi\co T\to S$ be a continuous surjective map of profinite spaces.
\begin{enumerate}
\item For $A\in\Rho\text{-}\mod^\mathrm{dis}$, we let $\prod_{T/S}A_T:=\pi_\ast(A_T)\in\cE_S^\mathrm{dis}(\Rho)$.
\item For $A\in\Rho\text{-}\mod^\mathrm{pro}$, we let $\bigoplus_{T/S}A_T:=\pi_\ast(A^\vee_T)^\vee\in\cE_S^\mathrm{pro}(\Rho)$.
\end{enumerate}
\end{definition}

We have the following general lemma:

\begin{lemma}\label{restrictionclosed}For $M$ a discrete set and $C$ a closed subset of a profinite space $T$, 
let $\{U_\alpha\}_{\alpha\in\Alpha}$ be a fundamental system of clopen neighborhoods of $C$ in $T$. We then have:
\[\cC(C,M)=\colim_{\alpha\in\Alpha}\cC(U_\alpha,M).\]
\end{lemma}

\begin{proof}For $T$ a discrete finite set, the lemma is trivial. Let $\{T_\beta\}_{\beta\in\Beta}$ be a cofiltered diagram of finite discrete set 
such that $T=\lim_{\beta\in\Beta}T_{\beta}$. Let $C_{\beta}$ and $U_{\alpha\beta}$, for $\alpha\in\Alpha$, be the image of $C$ and $U_\alpha$, 
respectively, in $T_{\beta}$, for $\beta\in\Beta$. We have $C=\lim_{\beta\in\Beta}C_{\beta}$ and $U_\alpha=\lim_{\beta\in\Beta}U_{\alpha\beta}$.
In particular, there holds:
\[\cC(C,M)=\colim_{\beta\in\Beta}\cC(C_\beta,M)\hspace{0.4cm}\mbox{and}\hspace{0.4cm}\cC(U_\alpha,M)=\colim_{\beta\in\Beta}\cC(U_{\alpha\beta},M),
\mbox{ for all }\alpha\in\Alpha.\]

Note also that, for a fixed $\beta\in\Beta$, the set $\{U_{\alpha\beta}\}_{\alpha\in\Alpha}$ is a fundamental system of clopen neighborhoods 
of $C_\beta$ in $T_\beta$. Since filtered colimits commute with filtered colimits, we have:
\[\begin{split}
\colim_{\alpha\in\Alpha}\cC(U_\alpha,M)&=\colim_{\alpha\in\Alpha}\colim_{\beta\in\Beta}\cC(U_{\alpha\beta},M)=\\
&=\colim_{\beta\in\Beta}\colim_{\alpha\in\Alpha}\cC(U_{\alpha\beta},M)=\colim_{\beta\in\Beta}\cC(C_\beta,M)=\cC(C,M).
\end{split}\]
\end{proof}

We can now prove:

\begin{proposition}\label{relativestalk}Let $\pi\co T\to S$ be a continuous surjective map of profinite spaces.
\begin{enumerate}
\item For $A\in\Rho\text{-}\mod^\mathrm{dis}$, there is a natural isomorphism $(\prod_{T/S}A_T)_s\cong A^{\pi^{-1}(s)}$, for all $s\in S$.
\item For $A\in\Rho\text{-}\mod^\mathrm{pro}$, there is a natural isomorphism $(\bigoplus_{T/S}A_T)_s\cong A[[\pi^{-1}(s)]]$, for all $s\in S$.
\end{enumerate}
\end{proposition}

\begin{proof}By Theorem~\ref{Pontryagindualityspaces} and Proposition~\ref{dualproduct}, it is enough to prove 
the first item of the proposition which immediately follows from Lemma~\ref{restrictionclosed}.
\end{proof}

The following theorem shows that we can associate a natural decomposition of free products
and sums to every continuous surjective map of profinite spaces:

\begin{theorem}\label{freedec}Let $\pi\co T\to S$ be a continuous surjective map of profinite spaces.
\begin{enumerate}
\item For $A\in\Rho\text{-}\mod^\mathrm{dis}$, there is a natural isomorphism: \[\prod_S\prod_{T/S}A_T\cong\prod_T A_T.\]
\item For $A\in\Rho\text{-}\mod^\mathrm{pro}$, there is a natural isomorphism: \[\bigoplus_S\bigoplus_{T/S}A_T\cong\bigoplus_T A_T.\]
\end{enumerate}
\end{theorem}

\begin{proof}(i): Let $p\co S\to\{\ast\}$ be the map to the one-point space. Then, by Proposition~\ref{H^0}, we have 
$\prod_T A_T=(p\circ\pi)_\ast(\ul{A})(\{\ast\})$ and $\prod_S\prod_{T/S}A_T=p_\ast(\pi_\ast(\ul{A}))(\{\ast\})$. 
Thus, the claim follows from the functoriality of push-forward of sheaves, i.e.\ $(p\circ\pi)_\ast=p_\ast\circ\pi_\ast$.
\smallskip

\noindent
(ii): By Proposition~\ref{dualproduct}, there is a natural isomorphism $(\prod_T A_T^\vee)^\vee\cong\bigoplus_T A_T$
and a series of natural isomorphisms:
\[(\prod_S\prod_{T/S}A_T^\vee)^\vee\cong\bigoplus_S(\prod_{T/S}A_T^\vee)^\vee\cong\bigoplus_S\bigoplus_{T/S}A_T.\]
The conclusion then follows from the previous item.
\end{proof}


\section{A Lannes-Quillen Theorem for mod-$p$ cohomology of profinite groups}
\subsection{A generalization of Theorem~\ref{Lannesfinite} to arbitrary profinite groups}
Let $\hom(V,G)$ be the profinite $G$-set of homomorphisms from an elementary abelian $p$-group $V$ to the profinite group $G$
and let $\Rep(V,G):=\hom(V,G)/G$ be the set of $G$-conjugacy classes. This is also a profinite space and the natural projection
$\pi\co\hom(V,G)\to\Rep(V,G)$ is continuous. 
 
Let $\cC(\hom(V,G),\F_p)$ be the set of continuous maps from the profinite set $\hom(V,G)$ to the finite discrete set $\F_p$.
This has a natural structure of (continuous) $\F_p[[G]]$-module.
In the proof of \cite[Proposition~4.1]{Symonds}, Symonds shows that there is a natural isomorphism:
\begin{equation}\label{Sym}
T_V H^\bullet(G)\cong H^\bullet(G;\cC(\hom(V,G),\F_p)).
\end{equation}

Thanks to the results of the previous sections, we can deduce from Symonds' isomorphism~\eqref{Sym} 
a version of Theorem~\ref{Lannesfinite} for profinite groups:

\begin{theorem}\label{Lannesprofinite}For $G$ a profinite group and $V$ an elementary abelian $p$-group, 
let us denote by $\pi\co\hom(V,G)\to\Rep(V,G)$ the orbit map.
There is then an \'etal\'e space $(F,q)$ of discrete $\F_p[[G]]$-modules over the profinite space $\Rep(V,G)$, 
whose fiber $F_{[\rho]}$, for all $[\rho]\in\Rep(V,G)$, is naturally isomorphic to the coinduced $G$-module 
$\cC(\pi^{-1}([\rho]),\F_p)\cong\coind^G_{C_G(\rho(V))}\F_p$, so that there holds:
\begin{equation}\label{inducedmod}
\cH^\bullet(G;F)_{[\rho]}=H^\bullet(G;F_{[\rho]})=H^\bullet(G;\cC(\pi^{-1}([\rho]),\F_p))\cong H^\bullet(C_G(\rho(V));\F_p),
\end{equation}
and there is a natural isomorphism of discrete $\F_p$-vector spaces:
\begin{equation}\label{Lannesisopro}
T_V H^\bullet(G)\cong\prod_{\Rep(V,G)}\cH^\bullet(G;F)=
\sideset{}{^{\mathrm{top}}}\prod_{\hspace{0.4cm}[\rho]\in\Rep(V,G)}H^\bullet(G;\cC(\pi^{-1}([\rho]),\F_p)).
\end{equation}
\end{theorem}

\begin{remark}\label{moreprecise} $\cH^\bullet(G;F)$ is an \'etal\'e space of unstable algebras over the mod-$p$ Steenrod algebra. 
By the universal property of the product, its product $\prod_{\Rep(V,G)}\cH^\bullet(G;F)$ then inherits such a structure,
so that \eqref{Lannesisopro} is an isomorphisms of unstable algebras over the mod-$p$ Steenrod algebra.
\end{remark}

\begin{proof}By Proposition~\ref{freesum} and item (ii) of Theorem~\ref{freedec}, there is a natural isomorphism of profinite $\F_p[[G]]$-modules:
\begin{equation}\label{orbitdec} 
\F_p[[\hom(V,G)]]\cong\bigoplus_{\Rep(V,G)}\bigoplus_{\hspace{0.4cm}\hom(V,G)\left/\Rep(V,G)\right.}(\F_p)_{\hom(V,G)}.
\end{equation}

 Put, for simplicity, $ \F_p[[\hom(V,G)\left/\Rep(V,G)\right.]]:=\bigoplus_{\hom(V,G)\left/\Rep(V,G)\right.}(\F_p)_{\hom(V,G)}$.
 By Proposition~\ref{relativestalk}, for all $[\rho]\in\Rep(V,G)$, there holds: 
 \[\begin{split}
 \F_p[[\hom(V,G)\left/\Rep(V,G)\right.]]_{[\rho]}&=\F_p[[\pi^{-1}([\rho])]]\cong\\
 &\cong\F_p[[G/C_G(\rho(V))]]=\ind^G_{C_G(\rho(V))}\F_p.
 \end{split}\]

Thus, $F:=\chom_{\F_p}(\F_p[[\hom(V,G)\left/\Rep(V,G)\right.]],(\F_p)_{\Rep(V,G)})$ is an an \'etal\'e space of discrete 
$\F_p[[G]]$-modules over the profinite space $\Rep(V,G)$, with the property that, for all $[\rho]\in\Rep(V,G)$, there holds:
\[\begin{split}
F_{[\rho]}&=\chom_{\F_p}(\F_p[[\hom(V,G)\left/\Rep(V,G)\right.]],(\F_p)_{\Rep(V,G)})_{[\rho]}=\\
&=\hom_{\F_p}(\F_p[[\hom(V,G)\left/\Rep(V,G)\right.]]_{[\rho]},\F_p)\cong\coind^G_{C_G(\rho(V))}\F_p.
 \end{split}\]
Then:
\[\cH^\bullet(G;F)_{[\rho]}=H^\bullet(G;F_{[\rho]})\cong H^\bullet(G;\coind^G_{C_G(\rho(V))}\F_p)\cong H^\bullet(C_G(\rho(V));\F_p),\]
which proves the first part of the theorem.

By the isomorphism~\eqref{orbitdec} and (ii) of Theorem~\ref{compatibility}, there is a natural isomorphism:
\[\hom_{\F_p}( \F_p[[\hom(V,G)]],\F_p)\cong \prod_{\Rep(V,G)}\chom_{\F_p}( \F_p[[\hom(V,G)\left/\Rep(V,G)\right.]],(\F_p)_{\Rep(V,G)}).\]
By (i) of Theorem~\ref{compatibility} and the fact that the functor $\prod_{\Rep(V,G)}$ is exact (cf.\ Corollary~\ref{prodexact}),
there is then a series of natural isomorphisms:
\[\begin{split}
&H^\bullet(G;\hom_{\F_p}( \F_p[[\hom(V,G)]],\F_p)\cong\\
&H^\bullet(G;\prod_{\Rep(V,G)}\chom_{\F_p}( \F_p[[\hom(V,G)\left/\Rep(V,G)\right.]],(\F_p)_{\Rep(V,G)}))\cong\\
&\cong\prod_{\Rep(V,G)}\cH^\bullet(G;\chom_{\F_p}(\F_p[[\hom(V,G)\left/\Rep(V,G)\right.]],(\F_p)_{\Rep(V,G)})).
\end{split}\]
This isomorphism, together with the isomorphism~\eqref{Sym}, now implies the isomorphism~\eqref{Lannesisopro}.
\end{proof}

Let us denote by $\hom(V,G)_i$ the subset of $\hom(V,G)$ consisting of the homomorphisms whose image has rank $i\leq\rk V$
and let $\Rep(V,G)_i$ be the image of $\hom(V,G)_i$ in $\Rep(V,G)$, for $i=0,\dots,\rk V$. It is easy to check that $\hom(V,G)_i$ and 
$\Rep(V,G)_i$ are clopen subsets of the profinite spaces $\hom(V,G)$ and $\Rep(V,G)$, respectively, so that there is a decomposition 
in disjoint clopen subsets:
\[\Rep(V,G)=\coprod_{i=0}^{\rk V}\Rep(V,G)_i.\]
By Theorem~\ref{Lannesprofinite}, Corollary~\ref{proddecomposition} and the definition of the reduced $T$-functor (cf.\ Section~\ref{defreduced}), 
we then have a series of natural isomorphisms of unstable modules:
\[H^\bullet(G)\oplus\bT_V H^\bullet(G)\cong T_V H^\bullet(G)\cong H^\bullet(G)\oplus\prod_{i=1}^{\rk V}
\sideset{}{^{\mathrm{top}}}\prod_{\hspace{0.4cm}[\rho]\in\Rep(V,G)_i}H^\bullet(G;\cC(\pi^{-1}([\rho]),\F_p)).\]

Since, in the above series of isomorphisms, $H^\bullet(G)$ is mapped identically onto itself, it follows:

\begin{corollary}\label{Lannesprofinitereduced}With the notations of Theorem~\ref{Lannesprofinite}, there is a natural isomorphism:
\[\bT_V H^\bullet(G)\cong\prod_{i=1}^{\rk V}
\sideset{}{^{\mathrm{top}}}\prod_{\hspace{0.4cm}[\rho]\in\Rep(V,G)_i}H^\bullet(G;\cC(\pi^{-1}([\rho]),\F_p)).\]
\end{corollary}

Note that, for a Lannes group $G$, there is, in particular, an isomorphism of $p$-Boolean algebras (cf.\ \cite[Theorem~3.10.2]{Schwartz}):
\begin{equation}\label{T^0}
T^0_V H^\bt(G)\cong\map(\Rep(V,G),\F_p),
\end{equation}
where, for two given sets $X$ and $Y$, we denote by $\map(X,Y)$ the set of maps from $X$ to $Y$.
In the case of a general profinite group $G$, we need instead to restrict to continuous maps:

\begin{corollary}\label{degree0iso}For a profinite group $G$ and an elementary abelian $p$-group $V$, there is a natural isomorphism
of discrete $\F_p$-vector spaces:
\[T^0_V H^\bt(G)\cong\cC(\Rep(V,G),\F_p)={\F_p}^{\Rep(V,G)}.\]
\end{corollary}

\begin{proof}This is an immediate consequence of Theorem~\ref{Lannesprofinite} and Corollary~\ref{freeproddescr}, but we can
also deduce it directly from Symonds' isomorphism~\eqref{Sym} as follows. 

By~\eqref{Sym}, there is a natural isomorphism $T^0_V H^\bt(G)\cong\cC(\hom(V,G),\F_p)^G$.
We then just need to describe the space of $G$-invariants $\cC(\hom(V,G),\F_p)^G$ more explicitly.
Let $G$ be the cofiltered limit of a diagram of finite groups $\{G_\alpha\}_{\alpha\in\Alpha}$. Note that
$\hom(V,G_\alpha)$ and $\Rep(\hom(V,G_\alpha))$ are discrete finite sets.
There is then a series of natural isomorphisms:
\[\begin{split}
\cC(\hom(V,G),\F_p)^G&\cong(\colim\limits_{\alpha\in\Alpha}\cC(\hom(V,G_\alpha),\F_p))^G\cong
\colim\limits_{\alpha\in\Alpha}\cC(\hom(V,G_\alpha),\F_p)^G\cong\\
&\cong\colim\limits_{\alpha\in\Alpha}\cC(\Rep(\hom(V,G_\alpha)),\F_p)\cong\cC(\Rep(V,G),\F_p).
\end{split}\]
\end{proof}

\section{Finite $p$-subgroups of Lannes groups}
In this section, we apply the results of the previous section to the problem of conjugacy separability of $p$-torsion elements
and finite $p$-subgroups of a Lannes group $G$ (cf.\ Definition~\ref{Lannesgroup}). For a given group $G$, we denote by 
$\cS_p(G)$ the category of finite $p$-subgroups of $G$, where the morphisms are induced by inclusion and conjugation by elements of $G$. 

A morphism of groups $f\co G\to L$ determines a functor $f_p\co\cS_p(G)\to\cS_p(L)$. The latter is
fully faithful if and only if the following two conditions are satisfied:
\begin{itemize}
\item $f_p$ is essentially injective, that is to say, given two finite $p$-subgroups $A_1$ and $A_2$ of $G$, the subgroups
$f(A_1)$ and $f(A_2)$ are conjugated in $L$ if and only if $A_1$ and $A_2$ are conjugated in $G$;
\item given a finite $p$-subgroup $A$ of $G$, the natural homomorphism $N_G(A)\to N_L(f(A))$ induces an isomorphism between the images 
of the homomorphisms, induced by conjugation, $N_G(A)\to\aut(A)$ and $N_L(f(A))\to\aut(f(A))$.
\end{itemize}

Since $\cS_p(G)$ and $\cS_p(L)$ are small categories, the functor $f_p\co\cS_p(G)\to\cS_p(L)$ is an equivalence if, 
besides the above two conditions, it satisfies the condition:
\begin{itemize}
\item $f_p$ is essentially surjective, that is to say, every  finite $p$-subgroup $B$ of $L$ is conjugated 
to the image $f(A)$ of an elementary abelian $p$-subgroup $A$ of $G$. 
\end{itemize}

We then let $\cA_p$ be the category of elementary abelian $p$-groups and let $\cA_p(G)$ be the full subcategory of $\cS_p(G)$ 
whose objects are the elementary abelian $p$-subgroups of $G$.

Note that if $L$ is a profinite group, the object sets of $\cS_p(L)$ and $\cA_p(L)$ (resp.\ of their quotients by conjugation $\cS_p(L)/L$ and 
$\cA_p(L)/L$) are all naturally endowed with the topology induced by their embeddings in the profinite space of closed subgroups of $L$ 
(resp.\ the profinite space of conjugacy classes of closed subgroups of $L$). 

The following lemma is probably well known to experts (at least, it is used implicitly in \cite{Symonds}). 
We include it for lack of suitable references and convenience of the reader:

\begin{lemma}\label{Rep}Let $f_V\co\Rep(V,G)\to\Rep(V,L)$, for $V\in\cA_p$, be the natural map induced by the homomorphism $f\co G\to L$. There holds:
\begin{enumerate}
\item The functor $f_p\co\cA_p(G)\to\cA_p(L)$ is fully faithful if and only if the natural map $f_V\co\Rep(V,G)\to\Rep(V,L)$ is injective
for all $V\in\cA_p$.
\item If $L$ is a profinite group, then the functor $f_p\co\cA_p(G)\to\cA_p(L)$ has dense image if and only if 
the map $f_V\co\Rep(V,G)\to\Rep(V,L)$ has dense image, for all $V\in\cA_p$.
\end{enumerate}
\end{lemma}

\begin{proof}(i): The forward implication is immediate. Let us then show that, if the map $f_V\co\Rep(V,G)\to\Rep(V,L)$ is injective, 
for all $V\in\cA_p$, the functor $f_p\co\cA_p(G)\to\cA_p(L)$ is fully faithful. It is clear that the injectivity of $f_V$, for all $V\in\cA_p$, 
implies that the functor $f_p\co\cA_p(G)\to\cA_p(L)$ is essentially injective. Let us then show that the injectivity of $f_V$, for all $V\in\cA_p$, 
implies that $f_p$ also satisfies the second condition above and so is fully faithful. 

Let $O_{[\rho]}$ and $O_{f_V([\rho])}$ be, respectively, the $\aut(V)$-orbits of an element 
$[\rho]\in\Rep(V,G)$ and of $f_V([\rho])\in\Rep(V,L)$. Since the map $f_V\co\Rep(V,G)\to\Rep(V,L)$ is $\aut(V)$-equivariant, 
by restriction, we get a natural surjective map $f_{V,[\rho]}\co O_{[\rho]}\to O_{f_V([\rho])}$. By hypothesis, this map is also injective and so bijective.

Let us denote by $\aut(V)\cdot\rho$ and $\aut(V)\cdot (f\circ\rho)$, respectively, the $\aut(V)$-orbits of $\rho\in\hom(V,G)$ and of $ f\circ\rho\in\hom(V,L)$. 
and let $\eta\co N_G(\rho(V))\to\aut(\rho(V))$ and $\mu\co N_L(f(V))\to\aut(f(\rho(V)))$ be the natural representations induced by conjugation.
We then have:
\[O_{[\rho]}\cong\aut(V)\cdot\rho\left/\eta(N_G(\rho(V)))\right.\hspace{0.2cm}\mbox{and}
\hspace{0.2cm}O_{f_V([\rho])}\cong\aut(V)\cdot (f\circ\rho)\left/\mu(N_L(f(\rho(V)))\right..\]

Let $\rho\co V\to G$ be an injective homomorphism. Then, from the injectivity of $f_V$, it easily follows that the composition $f\circ\rho\co V\to L$ 
is also injective. This implies that $\aut(V)$ acts with trivial stabilizers on both $\rho\in\hom(V,G)$ and $f\circ\rho\in\hom(V,L)$ and that $\eta(N_G(\rho(V)))$ 
and $\mu(N_L(f(\rho(V)))$ act freely on the respective orbits. 
Since the natural map $f_{V,[\rho]}\co O_{[\rho]}\to O_{f_V([\rho])}$ is bijective, it follows that 
the natural homomorphism of finite groups $\eta(N_G(\rho(V)))\to\mu(N_L(f(\rho(V)))$ is bijective and then an isomorphism.
This is enough to prove our claim about the functor $f_p$.
\smallskip

\noindent
(ii): For a group $G$ (abstract or profinite) and $k\in\N$, let $\cA_p^{\leq k}(G)$ be the full subcategory of $\cA_p(G)$ whose objects are 
the elementary abelian $p$-subgroups of $G$ of rank $\leq k$. Let us observe that, if $L$ is a profinite group, then the object set of
$\cA_p^{\leq k}(L)$ is a profinite space and so, in particular, a compact subset of the profinite space of all closed subgroups of $L$. 
The object set of $\cA_p(L)$ is then the union of all the object sets of the subcategories $\cA_p^{\leq k}(G)$, for $k\geq 0$, 
it is endowed with the associated final (colimit) topology and, for this topology, the object set of $\cA_p^{\leq k}(G)$ is a clopen subset. 

Let us denote by $f_p^{\leq k}\co\cA_p^{\leq k}(G)\to\cA_p^{\leq k}(L)$ the corresponding restriction of $f_p$. 
It is clear that the map $f_V\co\Rep(V,G)\to\Rep(V,L)$ has dense image if and only if the natural map $\tilde f_V\co\hom(V,G)\to\hom(V,L)$ has 
dense image. In order to prove the second item of the lemma, it is then enough to show that the map $\tilde f_V\co\hom(V,G)\to\hom(V,L)$ has 
dense image if and only if, for $k=\rk(V)$, the functor $f_p^{\leq k}\co\cA_p^{\leq k}(G)\to\cA_p^{\leq k}(L)$ has dense image. 

There are maps $\iota_V\co\hom(V,G)\to\Ob(\cA_p^{\leq k}(G))$ and $\iota_V'\co\hom(V,L)\to\Ob(\cA_p^{\leq k}(L))$, defined assigning 
to a homomorphism its image, such that $\iota_V'$ is continuous and there holds $\iota_V'\circ\tilde f_V=f_p^{\leq k}\circ\iota_V$. 
This immediately implies that, if $\tilde f_V$ has dense image, then $f_p^{\leq k}$ has also dense image. The other implication
follows from the observation that both maps $\iota_V$ and $\iota_V'$ have finite fibers and that the restriction of $\tilde f_V$ on these
fibers is surjective. Then, the subset $\Im\tilde f_V$ of $\hom(V,L)$ is dense if and only if its image in $\Ob(\cA_p^{\leq k}(L))$ by $\iota_V'$,
which coincides with $\Im f_p^{\leq k}$, is dense.
\end{proof}

Let us recall that a continuous homomorphism $f\co G\to L$ between topological groups, for $M$ a topological $L$-module, induces a 
homomorphism $f^\ast\co C^\bt(L,M)\to C^\bt(G,M)$ between the associated complexes of continuous cochains and then a homomorphism 
on continuous cohomology $f^\ast\co H^i(L;M)\to H^i(G;M)$. In case $G$ is an abstract group, we may endow $G$ with the discrete topology, 
so that its continuous cohomology coincides with the ordinary cohomology of the group and every homomorphism $f\co G\to L$ is continuous.
In what follows, we work, as usual, with constant $\F_p$-coefficients and assume that all homomorphisms are continuous. 
The main result of this section is then the following:

\begin{theorem}\label{separability}Let $f\co G\to L$ be a homomorphism from a Lannes group to a profinite group. We have: 
\begin{enumerate}
\item If the map induced on cohomology $f^\ast\co H^i(L;\F_p)\to H^i(G;\F_p)$ is surjective for $i\gg 0$, then the 
induced functor $f_p\co\cA_p(G)\to\cA_p(L)$ is fully faithful.
\item If the map induced on cohomology $f^\ast\co H^i(L;\F_p)\to H^i(G;\F_p)$ is injective for $i\gg 0$, then the 
induced functor $f_p\co\cA_p(G)\to\cA_p(L)$ has dense image.
\item If the map induced on cohomology $f^\ast\co H^i(L;\F_p)\to H^i(G;\F_p)$ is an isomorphism for $i\gg 0$ and
the group $G$ has only finitely many conjugacy classes of elementary abelian $p$-subgroups of fixed order $p^k$, for all $k\geq 0$, then the 
induced functor $f_p\co\cA_p(G)\to\cA_p(L)$ is an equivalence of categories.
\end{enumerate}
\end{theorem}

\begin{remark}The proof of Theorem~\ref{separability} only uses the fact that the group $G$ satisfies the isomorphism~\eqref{T^0}.
However, we do not know any example of a group, which is not a Lannes group, but for which this isomorphism holds. So we do not know
whether asking that only this condition be satisfied by $G$ would yield a substantial generalization of the theorem.
\end{remark} 

The proof is based on the following simple lemma:

\begin{lemma}\label{degree0}Let $f\co G\to L$ be a homomorphism from a Lannes group to a profinite group such that
the map induced on cohomology $f^\ast\co H^i(L;\F_p)\to H^i(G;\F_p)$ is surjective (resp.\ injective) for $i\gg 0$, then,
for every elementary abelian $p$-group $V$, the induced map $T_V^0 H^\bt(L)\to T_V^0 H^\bt(G)$ is surjective (resp.\ injective).
\end{lemma}

\begin{proof}The hypotheses of the lemma imply that the cokernel (resp.\ kernel) of the map of unstable modules $f^\ast\co H^\bt(L;\F_p)\to H^\bt(G;\F_p)$ 
is a nilpotent unstable module $N$. By \cite[Lemma~3.10.1]{Schwartz}, we then have $T_V^0 N=\{0\}$ and the conclusion follows from the
exactness of the functor $T_V^0$. 
\end{proof}

\begin{proof}[Proof of Theorem~\ref{separability}](i): By Corollary~\ref{degree0iso} and~\eqref{T^0}, there are natural isomorphisms
of $\F_p$-vector spaces $T^0_V H^\bt(L)\cong\cC(\Rep(V,L),\F_p)$ and $T^0_V H^\bt(G)\cong\map(\Rep(V,G),\F_p)$.

By hypothesis and Lemma~\ref{degree0}, the map $f_V^\ast\co\cC(\Rep(V,L),\F_p)\to\map(\Rep(V,G),\F_p)$, induced by the
map $f_V\co\Rep(V,G)\to\Rep(V,L)$, is then surjective. 

If the map $f_V$ is not injective, it is easy to construct a map $\Rep(V,G)\to\F_p$ not in the image of the map $f_V^\ast$. 
Therefore, $f_V$ is injective. By the first item of Lemma~\ref{Rep}, this implies that the functor $f_p\co\cA_p(G)\to\cA_p(L)$ is fully faithful. 
\smallskip

\noindent
(ii): By hypothesis and Lemma~\ref{degree0}, the map $f_V^\ast\co\cC(\Rep(V,L),\F_p)\to\map(\Rep(V,G),\F_p)$ is injective. 
The profinite set $\Rep(V,L)$ is the inverse limit of the finite sets $\Rep(V,L')$, for $L'$ a finite quotient of $L$. Therefore,
$f_V\co\Rep(V,G)\to\Rep(V,L)$ has dense image if and only if for every finite quotient map $\pi\co L\to L'$, with associated map 
$\pi_V\co\Rep(V,L)\to\Rep(V,L')$, the composition $\pi_V\circ f_V\co\Rep(V,G)\to\Im\pi_{V}$ is surjective.
 
If $f_V\co\Rep(V,G)\to\Rep(V,L)$ has not dense image, there is then a finite quotient map $\pi\co L\to L'$ such that the composition 
$\pi_V\circ f_V\co\Rep(V,G)\to\Im\pi_{V}$ is not surjective. 
This implies that the induced map $f_V^\ast\circ\pi_V^\ast\co\map(\Im\pi_V,\F_p)\to\map(\Rep(V,G),\F_p)$ has a nontrivial kernel
and then, since the finite discrete $\F_p$-vector space $\map(\Im\pi_V,\F_p)$ identifies with a subspace of the discrete $\F_p$-vector space 
$\cC(\Rep(V,L),\F_p)$, it implies that the map $f_V^\ast\co\cC(\Rep(V,L),\F_p)\to\map(\Rep(V,G),\F_p)$ is not injective, in contradiction
with the hypothesis. Therefore, $f_V$ has dense image and the conclusion of the theorem follows from the second item of Lemma~\ref{Rep}.
\smallskip

\noindent
(iii): Since the only dense subset of a discrete set is the all set, our hypotheses and the second item of Theorem~\ref{separability} imply 
that the induced functor $f_p/G\co\cA_p(G)/G\to\cA_p(L)/G$ is surjective on objects. In turn, this implies that the functor $f_p\co\cA_p(G)\to\cA_p(L)$ 
is essentially surjective. By hypothesis and the first item of Theorem~\ref{separability}, $f_p$ is also fully faithful. 
Therefore, the conclusion follows.
\end{proof}

Theorem~\ref{separability} implies the following conjugacy separability criterion for Lannes groups which improves \cite[Theorem~1.5]{MZ}:

\begin{corollary}\label{separability2}Let $G$ be a Lannes group such that the natural homomorphism $\iota\co G\to\wh{G}$ to its profinite completion 
induces a surjective map $\iota^\ast\co H^i(\wh{G};\F_p)\to H^i(G;\F_p)$ for $i\gg 0$. Then, $p$-elements are conjugacy distinguished and elementary 
abelian $p$-subgroups of $G$ are subgroup conjugacy distinguished.
\end{corollary}

With more restrictive hypotheses, we have the following analogue of Symonds' theorem \cite[Theorem~1.1]{Symonds} (see
also \cite[Theorem~A]{BZ}, for a similar result):

\begin{theorem}\label{separability3}Let $G$ be a Lannes group such that the cohomology ring $H^\bt(G;\F_p)$ is a finitely generated $\F_p$-algebra and 
let $f\co G\to L$ be a homomorphism to a profinite group such that the induced map $f^\ast\co H^i(L;\F_p)\to H^i(G;\F_p)$ is an isomorphism for $i\gg 0$. 
Then, the induced functor $f_p\co\cS_p(G)\to\cS_p(L)$ is an equivalence of categories.
\end{theorem}

Let $\Rep(V,G)^\ast$ be the set of \emph{nontrivial} homomorphisms from the elementary abelian $p$-group $V$ to $G$.
Note that, if $G$ is a profinite group, then $\Rep(V,G)^\ast$ is a clopen subset of $\Rep(V,G)$. In particular, it is a profinite space.
We will need the following lemma:

\begin{lemma}\label{inducediso}Let $f\co G\to L$ be a homomorphism from a Lannes group to a profinite group such that
the induced map $f^\ast\co H^i(L;\F_p)\to H^i(G;\F_p)$ is an isomorphism for $i\gg 0$. 
Then, $f$ induces an isomorphism of unstable modules:
\[\sideset{}{^{\mathrm{top}}}\prod_{[\rho]\in\Rep(V,L)^\ast}H^\bullet(L;\cC(\pi_L^{-1}([\rho]),\F_p))\sr{\sim}{\to}
\prod_{[\mu]\in\Rep(V,G)^\ast}H^\bullet(G;\cC(\pi_G^{-1}([\mu]),\F_p)),\]
where $\pi_L\co\hom(V,L)^\ast\to\Rep(V,L)^\ast$ and $\pi_G\co\hom(V,G)^\ast\to\Rep(V,G)^\ast$ are the natural orbit maps.
\end{lemma}

\begin{proof}This is essentially the same proof as \cite[Corollary~3.6]{Symonds}. 
The hypothesis implies that there is an exact sequence of unstable modules:
\[0\to K^\bt\to H^\bt(L;\F_p)\to H^\bt(G;\F_p)\to C^\bt\to 0,\]
where both $K^\bt$ and $C^\bt$ are bounded and so locally finite. Applying the $T$-functor, we get an exact sequence of unstable modules
(cf.\ \cite[Proposition~3.3.6]{Schwartz}):
\[0\to K^\bt\to T_V H^\bt(L;\F_p)\to T_V H^\bt(G;\F_p)\to C^\bt\to 0.\]
The above two exact sequences then imply that $f$ induces an isomorphism after applying the reduced Lannes' $T$-functor 
$f^\ast\co\bT_V H^\bt(L;\F_p)\sr{\sim}{\to} \bT_V H^\bt(G;\F_p)$ and the
conclusion follows from Corollary~\ref{Lannesprofinitereduced} and the similar statement which holds for Lannes groups.
\end{proof}

\begin{proof}[Proof of Theorem~\ref{separability3}]The isomorphism~\eqref{Lannesiso} for Lannes groups implies, in particular, that the map 
from $\Rep(V,G)$ to $\hom_\cK(H^\bt(G),H^\bt(V))$, which sends a homomorphism $\rho\co V\to G$ to the induced map of unstable algebras 
$\rho^\ast\co H^\bt(G)\to H^\bt(V)$, is injective. From the hypothesis that $H^\bt(G)$ is finitely generated,
it follows that $\hom_\cK(H^\bt(G),H^\bt(V))$ and so $\Rep(V,G)$ is finite for every elementary abelian $p$-group $V$. 
This implies that all hypotheses of item (iii) in Theorem~\ref{separability} are satisfied. By (iii) of Theorem~\ref{separability}, 
we then conclude that the restriction $f_p|_{\cA_p(G)}\co\cA_p(G)\to\cA_p(L)$ is an equivalence of categories. In particular, 
$\Rep(V,L)=\Rep(V,G)$ (cf.\ Lemma~\ref{Rep}) is finite for every elementary abelian $p$-group $V$. 

We can now proceed by induction as in Symonds' proof of \cite[Theorem~1.1]{Symonds}. Hence,
by the same argument of Symonds (cf.\ \cite[Section~2]{Symonds}), it is enough to show that, if $f\co G\to L$ is a homomorphism 
which satisfies the hypotheses of the theorem, then, for any given homomorphism $\rho\co V\to G$ from an elementary abelian $p$-group, 
the induced homomorphism $f_\rho\co C_G(\rho(V))\to C_L(f(\rho(V)))$ induces an isomorphism on mod-$p$ cohomology.

As we remarked above, the hypotheses on $f$ imply that $\Rep(V,L)=\Rep(V,G)$, and so $\Rep(V,L)^\ast=\Rep(V,G)^\ast$, are finite discrete sets 
for every elementary abelian $p$-group $V$. Therefore, we have that $\prod^\mathrm{top}_{\Rep(V,L)^\ast}=\prod_{\Rep(V,L)^\ast}$ which, 
together with Lemma~\ref{inducediso}, implies the above claim on $f_\rho$.
\end{proof}

\end{document}